\documentclass[twoside]{amsart}
\usepackage{latexsym,amsfonts,amsmath,amssymb}
\usepackage{enumerate}
\usepackage{color}

\newtheorem{Th}{Theorem}[section]
\newtheorem{Rem}[Th]{Remark}
\newtheorem{Ex}[Th]{Example}
\newtheorem{Lemma}[Th]{Lemma}
\newtheorem{Def}[Th]{Definition}
\newtheorem{Prop}[Th]{Proposition}
\newtheorem{Cor}[Th]{Corollary}

\catcode`\@=11

\renewcommand{\section}%
   {\setcounter{equation}{0}\@startsection {section}{1}{\z@}{-3.5ex plus -1ex
  minus -.2ex}{2.3ex plus .2ex}{\Large\bf}}

\def\supp{\mathop{\rm supp}\nolimits}

\def\Im{\mathop{\rm Im}\nolimits}
\def\psh{\mathop{\rm psh}\nolimits}
\def\Ass{\mathop{\rm Ass}\nolimits}
\def\Ext{\mathop{\rm Ext}\nolimits}
\def\Ker{\mathop{\rm Ker}\nolimits}
\def\coker{\mathop{\rm coker}\nolimits}
\def\indlim{\mathop{\rm ind\,lim}}
\def\projlim{\mathop{\rm proj\,lim}}
\def\ch{\mathop{\rm ch}\nolimits}
\def\ds{\displaystyle}
\def\R{\mathbb R}
\def\C{\mathbb C}
\def\N{\mathbb N}

\newcommand{\D}{\mathcal{D}}
\newcommand{\E}{\mathcal{E}}
\newcommand{\F}{\mathcal{F}}
\newcommand{\M}{\mathcal{M}}
\newcommand{\W}{\mathcal{W}}
\newcommand{\I}{\mathcal{I}}
\newcommand{\Sch}{\mathcal{S}}

\newcommand{\beqsn}{\arraycolsep1.5pt\begin{eqnarray*}}
\newcommand{\eeqsn}{\end{eqnarray*}\arraycolsep5pt}
\newcommand{\beqs}{\arraycolsep1.5pt\begin{eqnarray}}
\newcommand{\eeqs}{\end{eqnarray}\arraycolsep5pt}

\title{The Overdetermined Cauchy Problem for  $\omega$-ultradifferentiable 
Functions}

\author[Boiti]{Chiara Boiti}
\address{
Dipartimento di Matematica e Informatica \\Universit\`a di Ferrara\\
Via Ma\-chia\-vel\-li n.~30\\
I-44121 Ferrara\\
Italy}
\email{chiara.boiti@unife.it\\
elisabetta.gallucci@student.unife.it}

\author[Gallucci]{Elisabetta Gallucci}
\begin{document}


\begin{abstract}
In this paper we study the Cauchy problem for overdetermined systems
of linear partial differential operators with constant coefficients in some
spaces of $\omega$-ultradifferentiable functions in the sense of Braun, Meise 
and Taylor \cite{BMT}, for non-quasianalytic weight functions $\omega$.
We show that existence of solutions of the Cauchy problem is equivalent to the
validity of a Phragm\'en-Lindel\"of principle for entire and plurisubharmonic 
functions on some irreducible affine algebraic varieties.
\end{abstract}

\maketitle
% computing hour and minutes
%  \newcount\minutes
%  \newcount\hour
%  \minutes=\time
%  \divide\minutes by 60
%  \hour=\minutes
%  \minutes=\time
%  \multiply \hour by 60
%  \advance \minutes by -\hour
%  \divide \hour by 60
%  \markboth{\today\space\number\hour :\number\minutes}%
%  {\today\space\number\hour :\number\minutes}
\markboth{\sc The overdetermined Cauchy problem\ldots }
{\sc C.~Boiti and E.~Gallucci}

  \section{Introduction}
In this paper we consider the Cauchy problem for overdetermined systems
of linear partial differential operators with constant coefficients
in some classes of $\omega-$ultradifferentiable functions, in the
sense of Braun, Meise and Taylor \cite{BMT}.

We consider
 a continuous
increasing weight function $\omega:\ [0,+\infty)\to[0,+\infty)$ satisfying
\beqsn
\hspace*{-5.3cm}
&&(\alpha)\qquad \exists K\geq1:\ \omega(2t)\leq 
K(1+\omega(t))\qquad\forall t\geq0\\
\hspace*{-5.3cm}
&&(\beta)\qquad \int_{1}^{\infty}
\frac{\omega(t)}{t^{2}}dt<\infty\\
\hspace*{-5.3cm}
&&(\gamma)'\qquad \exists\,a\in\mathbb{R},
\,\,b>0\,:\,\,\omega(t)\geq a+b\log(1+t)\qquad
\forall t\geq0\\
\hspace*{-5.3cm}
&&(\delta)\qquad\varphi:\ [0,\infty)\rightarrow[0,\infty),
~~~~~~~\varphi(t):=\omega(e^{t})\ 
\mbox{is convex.}
\eeqsn

With respect to weight functions considered in \cite{BMT},
we weakened their condition
\begin{itemize}
\item [{$(\gamma)$}] ${\displaystyle 
\underset{t\rightarrow\infty}{\lim}\frac{\log(1+t)}{\omega(t)}=0,}$
\end{itemize}
by condition $(\gamma)'$ above, in the spirit of the original paper
of Bj\"{o}rck \cite{Bj}. For this reason in Section~2 we briefly retrace
the paper of Braun, Meise and Taylor \cite{BMT}, defining the spaces
$\E_{\{\omega\}}$ and $\E_{(\omega)}$ of $\omega$-ultradifferentiable functions 
of Roumieu and Beurling type,
but enlightening the
results that are still valid with the weaker condition $(\gamma)'$
and those ones which need the stronger condition $(\gamma)$ 
(cf. also \cite{Fi,G}).
It comes out that
condition $(\gamma)$ is needed in the Roumieu case $\E_{\{\omega\}}$, while
condition $(\gamma)'$ is sufficient in the Beurling case $\E_{(\omega)};$
in particular, the space ${\mathcal E}(\Omega)$ of ${\mathcal C}^{\infty}$
functions on an open set $\Omega\subset\R^N$ can be viewed as
$\E_{(\omega)}(\Omega)$ for $\omega(t)=\log(1+t)$.
The utility of weakening condition $(\gamma)$ by condition $(\gamma)'$ is
clear, for instance, in the forthcoming paper \cite{BJO},  
for the description of the space
$\Sch_\omega$ of $\omega$-ultradifferentiable Schwartz functions.

In Section~3 we investigate the overdetermined Cauchy problem in the
Beurling case. To be more precise, we settle the Cauchy problem in
the frame of Whitney $\omega-$ultradifferentiable functions, 
in the spirit of \cite{N2,BN1,BN3}, in order
to bypass the question of formal coherence of the data, which naturally
arises in the overdetermined case.

Indeed, in the classical Cauchy problem for a linear partial differential 
equation with initial data on a hypersurface, smooth initial data together
with the equation allow to compute the Taylor series of a smooth solution
at any given point of the hypersurface.

This leads, in the case of systems of linear partial
differential equations, to the notion of formally
non-characteristic hypersurface that was considered in \cite{AHLM,AN2,N2}.

In the case of overdetermined systems, the question of the formal
coherence of the data would be particularly intricate, so that 
the above remarks suggest further generalizations of the Cauchy problem,
where the assumption that the initial data are given on a formally 
non-characteristic hypersurface is dropped, and we allow
formal solutions (in the sense of Whitney)
of the given system on any closed subset as initial data.

Using Whitney functions, we
can thus consider a more general framework
in which two quite arbitrary sets are involved. We take $K_{1}$ and
$K_{2}$ closed convex subsets of $\R^N$ with $K_{1}\subsetneq K_{2}$
and $\overline{\mathring{K_{j}}}=K_{j}$ for $j=1,\,2$, thinking
at $K_{1}$ as the set where the initial data are given, and at $K_{2}$
as the set where we want to find a solution of the following Cauchy
problem:
\begin{equation}
\begin{cases}
A_{0}(D)u=f\\
\left.u\right|_{K_{1}}\equiv\varphi,
\end{cases}\label{eq:acqua}
\end{equation}
where $A_{0}(D)$ is an $a_{1}\times a_{0}$ matrix of linear partial
differential operators with constant coefficients, 
$\varphi\in\left(W_{K_{1}}^{(\omega)}\right)^{a_{0}}$,
$f\in\left(W_{K_{2}}^{(\omega)}\right)^{a_{1}}$ are the given Cauchy
data in the Whitney classes of $\omega-$ultradifferentiable functions
of Beurling type on $K_{1}$ and $K_{2}$ respectively, and
$\left.u\right|_{K_{1}}\equiv\varphi$
means that they are equal in the Whitney sense, i.e. with all their
derivatives.

It comes out (see Section~3) that, in order to find a solution 
$u\in\left(W_{K_{2}}^{(\omega)}\right)^{a_{0}}$
of the Cauchy problem (\ref{eq:acqua}), the function $f$ must satisfy
some integrability conditions. These may be written as
\begin{equation}
\begin{cases}
A_{1}(D)f=0\\
\left.f\right|_{K_{1}}\equiv0,
\end{cases}\label{eq:fuoco}
\end{equation}
for a matrix $A_{1}(D)$ of linear partial differential operators
with constant coefficients obtained by a Hilbert resolution of 
$\M={\rm coker}\left(^{t}A_{0}(\zeta):{\mathcal P}^{a_{1}}
\rightarrow{\mathcal P}^{a_{0}}\right):$
\[
0\longrightarrow{\mathcal P}^{a_{d}}
\xrightarrow{^{t}A_{d-1}(\zeta)}{\mathcal P}^{a_{d-1}}
\longrightarrow\ldots\longrightarrow{\mathcal P}^{a_{2}}
\xrightarrow{^{t}A_{1}(\zeta)}{\mathcal P}^{a_{1}}
\xrightarrow{^{t}A_{0}(\zeta)}{\mathcal P}^{a_{0}}
\longrightarrow{\mathcal M}\longrightarrow0,
\]
where ${\mathcal P}=\mathbb{C}[\zeta_{1},\ldots,\zeta_{N}]$.

The rows of the matrix $A_{1}(D)$ give a system of generators for
the module of all integrability conditions for $f$ that can be expressed
in terms of partial differential operators,
and if $A_1(\zeta)\not\equiv0$ we say that the Cauchy problem is 
{\em overdetermined}.

We prove in Theorem \ref{thm:3.4.8} that the Cauchy problem (\ref{eq:acqua}),
for $f$ satisfying (\ref{eq:fuoco}), admits at least a solution
if and only if the following Phragm\'en-Lindel\"{o}f principle holds
for all $\wp\in\Ass(\M)$ and $V=V(\wp)$:
\[
(Ph-L)_{{\rm psh}}\begin{cases}
\forall\alpha\in\N,\,\,\exists\,\beta\in\N,\,\,C>0\,\,\mbox{such that} & 
\mbox{}\\
\mbox{if }u\in\psh(V)\,\,\mbox{satisfies for some costants }
\alpha_{u}\in\N,\,\,c_{u}>0\\
\begin{cases}
u(\zeta)\leq H_{K_{\alpha}^{2}}(\Im\zeta)+\alpha\omega(\zeta)\,,\qquad
\forall\zeta\in V\\
u(\zeta)\leq H_{K_{\alpha_{u}}^{1}}(\Im\zeta)+\alpha_{u}\omega(\zeta)+c_{u}\,,
\qquad\forall\zeta\in V
\end{cases}\\
\mbox{then it also satisfies:}\\
u(\zeta)\leq H_{K_{\beta}^{1}}(\Im\zeta)+\beta\omega(\zeta)+C\,,
\qquad\forall\zeta\in V,
\end{cases}
\]
where $H_{K}$ is the supporting function of the compact set $K$,
$\left\{ K_{\alpha}^{j}\right\} _{\alpha}$ is a sequence of compact
subsets of $K^{j}$ with $K_{\alpha}^{j}\subset\mathring{K}_{\alpha+1}^{j}$
and $K^{j}=\cup_\alpha K_{\alpha}^{j}$ for $j=1,\,2$, $V$ is the
{\em complex characteristic variety} of $\mathcal P/\wp$ defined by
\beqs
\label{Vwp}
V=V(\wp):=
\{\zeta\in\mathbb{C}^{N}:\,p(-\zeta)=0\ \,\forall p\in\wp\},
\eeqs
and $\psh(V)$ is the set of all plurisubharmonic functions on
$V$ (in the sense of Definition~\ref{defpsh}).

Relating the existence of solutions of the Cauchy problem to the validity 
of a Phragm\'en-Lindel\"of principle may be very useful. For instance,
in the case of a Gevrey weight $\omega(t)=t^{1/s}$ , $s>1$, it was found in
\cite{BM} a complete characterization of
algebraic curves V that satisfy the Phragm\'en-Lindel\"of principle, 
by means of Puiseux series expansions on the branches of V at infinity:
it comes out that the exponents and coefficients of the Puiseux series 
expansions are strictly related to the Gevrey order $s$. This implies that,
looking at the Puiseux series expansions at infinity of the complex 
characteristic variety associated to the system $A_0(D)$, we can establish 
in which (small) Gevrey classes the Cauchy problem admits at least a 
solution and in which classes it doesn't work. Since Puiseux series 
expansions can be computed by several computer programs, such as MAPLE 
for instance, this characterization may be very useful.

\section{Ultradifferentiable functions}

In the present section we follow \cite{BMT}, enlightening when
condition $(\gamma)$ below can be weakened by condition $(\gamma)'$ of
Definition \ref{def:2.1.2}.

We recall, from \cite{BMT}, the following:
\begin{Def}
\begin{em}
  \label{1.1}
  Let $\omega:\,[0,\infty)\rightarrow[0,\infty)$
      be a continuous increasing function.
      It will be called a \emph{non-quasianalytic
    weight function
  $\omega\in\mathcal{W}$} if it has the following properties:
  \begin{itemize}
    \item[$(\alpha)$]
$\exists$ $K\geq1:$ $\ \omega(2t)\leq K(1+\omega(t))$~~~~~
      $\quad\forall t\geq0$,
    \item[$(\beta)$]
      ${\displaystyle \int_{1}^{\infty}\frac{\omega(t)}{t^{2}}dt<\infty,}$
    \item[$(\gamma)$]
      ${\displaystyle \underset{t\rightarrow\infty}{\lim}
      \frac{\log(1+t)}{\omega(t)}=0,}$
    \item[$(\delta)$]
      $\varphi:[0,\infty)\rightarrow[0,\infty)$,~~~~~~~$\varphi(t)$
      :=$\omega(e^{t})$
        is convex.
        \end{itemize}
For $z\in\mathbb{C}^{N}$ we write $\omega(z)$ for $\omega(|z|)$,
where $|z|=\sum_{j=1}^{N}|z_{j}|$.
\end{em}
\end{Def}

\begin{Rem}
  \begin{em}
    Condition $(\beta)$ is the condition of non-quasianalyticity and it will
    ensure, in the following, the existence of functions with compact support
    (cf. Remark~\ref{rem218}).
  \end{em}
  \end{Rem}

It will be sometimes possible, and usefull (see \cite{BJO}), to weaken condition
$(\gamma)$ by the following:
\begin{Def}
\begin{em}
\label{def:2.1.2}Let $\omega:\ [0,\infty)\rightarrow[0,\infty)$ be
a continuous increasing function. It will be called a {\em non-quasianalytic
  weight function $\omega\in\W'$}
if it satisfies the conditions $(\alpha),\,(\beta),\,(\delta)$
and
\beqsn
\hspace*{-5.3cm}
\left(\gamma\right)'\qquad\exists\,a\in\mathbb{R},\,\,b>0\,:\,\,
\omega(t)\geq a+b\log(1+t),
\qquad\forall t\geq0.
\eeqsn
Set again $\omega(z)=\omega(|z|)$ for $z\in\mathbb{C}^{N}$.
\end{em}
\end{Def}

Then we
can define the \emph{Young conjugate} $\varphi^{*}$ of $\varphi$
by
\begin{align*}
\varphi^{*}:\mbox{[0,\ensuremath{\infty})} & \longrightarrow\R\\
y & \longmapsto\underset{x\geq0}{\sup}(xy-\varphi(x)).
\end{align*}

There is no loss of generality to assume that $\omega$ vanishes on
$[0,1]$ (cf. also \cite{AJO}). Then $\varphi^*$ has only non-negative
values, it is convex and increasing, satisfies $\varphi^{*}(0)=0$
and $(\varphi^{*})^{*}=\varphi$ (cf. \cite{BMT,BL}).
Moreover if ${\displaystyle \underset{x\rightarrow\infty}{\lim}
  \frac{x}{\varphi(x)}=0}$
then ${\displaystyle \underset{y\rightarrow\infty}{\lim}\frac{y}{\varphi^{*}
    (y)}=0}$.
Note that $(\gamma)$ implies ${\displaystyle
  \underset{x\rightarrow\infty}{\lim}\frac{x}{\varphi(x)}=0}$.

\begin{Ex}
\begin{em}
  The following functions $\omega\in\W'$ are examples of
  non-quasianalytic weight functions (eventually after a change in
  the interval $[0,\delta]$ for suitable $\delta>0$):
\beqs
  \label{eq:(1)-1}
  \omega(t)=&&t^{\alpha},\qquad0<\alpha<1\\
  \label{212}
  \omega(t)=&&(\log(1+t))^{\beta},\qquad\beta\geq1 \\
  \nonumber
\omega(t)=&&t(\log(e+t))^{-\beta}, \qquad\beta>1. 
\eeqs
They are all weight functions also in $\W$, except \eqref{212} for
$\beta=1$ which satisfies $(\gamma)'$ but not $(\gamma)$.
\end{em}
\end{Ex}

As in \cite{BMT}, the following lemmas are valid for $\omega\in\W'$
(see \cite{G} for more details):

\begin{Lemma}
\label{lemma12BMT}
Let $\omega\in\W'$. Then
\[
\omega(x+y)\leq K(1+\omega(x)+\omega(y)), \qquad\forall x,y\in\C^N.
\]
\end{Lemma}

\begin{Lemma}
  For $\omega\in\W'$ and $\varphi(t)=\omega(e^t)$, there exists
  $L>0$ such that
 \[
 \varphi^{*}(y)-y\geq L\varphi^{*}\left(\frac{y}{L}\right)-L,
 \qquad\forall y\geq0.
\]
\end{Lemma}

\begin{Lemma}
  For $\omega\in\W'$ and $\varphi(t)=\omega(e^t)$ we have that
  ${\displaystyle \frac{\varphi(x)}{x}}$ and ${\displaystyle
    \frac{\varphi^{*}(s)}{s}}$
are increasing.
\end{Lemma}

\begin{Lemma}
Let $\omega\in\W'$. Then there exists a weight function $\sigma\in\W$
with $\omega(t)=o(\sigma(t))$.
\end{Lemma}

On the contrary, for the following results of
\cite{BMT}, the condition
$\omega\in\W$ is needed:

\begin{Lemma}
  \label{1.7}
  Let $\omega\in\W$ and $g:\,[0,\infty)\rightarrow[0,\infty)$
satisfying $g(t)=o(\omega(t))$ as $t$ tends to $\infty$. Then there
exists $\sigma\in\W$ with the following properties:
\begin{itemize}
  \item[(i)]
$g(t)=o(\sigma(t))$ ~~~as $t\rightarrow\infty$,
\item[(ii)]
$\sigma(t)=o(\omega(t))$ ~~~as $t\rightarrow\infty$,
\item[(iii)]
  $\ds\forall A>1:\quad
  \limsup_{t\to\infty}\frac{\sigma(At)}{\sigma(t)}\leq
  \limsup_{t\to\infty}
  \frac{\omega(At)}{\omega(t)}$.
\end{itemize}
If, in addition, there is $R\geq1$ such that $\left.\omega\right|_{[R,\infty)}$
is concave, then it is possible to make $\left.\sigma\right|_{[R,\infty)}$
  concave, too.
\end{Lemma}

\begin{Rem}
  \begin{em}
  \label{remG1}
  $(1)$ If $\omega$ is a quasianalytic weight, i.e.
  $\omega\in\W$ is a weight function except that it doesn't
satisfy $(\beta)$, then the function $\sigma$ can be constructed
as above, except that it may not satisfy $(\beta)$.

$(2)$ We can construct $\sigma$ so that it coincides with $\omega$
on a given arbitrarily large bounded interval.%\label{(2)}

$(3)$ If $\omega,\rho\in\W$ are concave on
  $\left(R,\infty\right)$,
with $\rho=o(\omega)$, then for each $D>0$ there is a weight function
$\sigma$, as in Lemma \ref{1.7}, with the additional property that
$\sigma\equiv\rho$ on $\left[0,D\right]$.
\end{em}
\end{Rem}

\begin{Prop}
Let $\omega\in\W$ and, for $j\in\mathbb{N}$,
let $g_{j}:\, [0,+\infty)\to[0,+\infty)$ satisfying
    $g_{j}(t)=o(\omega(t))$ as $t$ tends
to $\infty$. Then there exists $\sigma\in\W$ with the
following properties:

$(i)$ ${\displaystyle {\displaystyle \underset{t\rightarrow\infty}{\lim}
    \frac{g_{j}(t)}{\sigma(t)}=0\,\,\,\,\,\,\forall j\in\mathbb{N},}}$

$(ii)$ ${\displaystyle \underset{t\rightarrow\infty}{\lim}{\displaystyle
    \frac{\sigma(t)}{\omega(t)}=0,}}$

$(iii)$ $\forall A>1:\quad
{\displaystyle \limsup_{t\to\infty}
  \frac{\sigma(At)}{\sigma(t)}\leq
  \limsup_{t\to\infty}
  \frac{\omega(At)}{\omega(t)}.}$
\end{Prop}

\begin{Lemma}
  Let $\omega\in\W$. Then there exists a nonzero function
  $g\in{\Sch}(\mathbb{R})$
with support in $(-\infty,0]$ for which the Fourier transform $\widehat{g}$
satisfies
\[
|\widehat{g}(x)|\leq  e^{ -2K\omega(x)}
\,\,\,\,\,\,\forall x\in\mathbb{R},
\]
where $K$ denotes the constant in $(\alpha)$.
\end{Lemma}

\begin{Prop}
  Let $\omega\in\W$. Then for every $\varepsilon>0$ there exists
  $h\in \mathcal C^{\infty}(\mathbb{R})$,
$h\not=0$, with
\begin{align*}
  & {\displaystyle \mbox{\ensuremath{\supp}}(h)}
  \subset[-\varepsilon,\varepsilon]\\
 & {\displaystyle \int_{-\infty}^{+\infty}|\hat{h}(x)|e^{\omega(x)}dx}<\infty.
\end{align*}
\end{Prop}

However, the following two propositions for the existence of functions
with compact support are valid also for $\omega\in\W'$:

\begin{Prop}
Let $\omega\in\W'$. Then for each $N\in\mathbb{N}$ there exists
$\delta_{N}>0$ such that for every $\varepsilon>0$ there exists
$H\in\mathcal C^{\infty}(\mathbb{R}^{N})$, $H\not=0$, with
\begin{align*}
  & {\displaystyle \mbox{\ensuremath{\supp}}(H)}
  \subset[-\varepsilon,\varepsilon]^{N}\\
  & {\displaystyle \int_{\mathbb{R}^{N}}
      |\hat{H}(T)|e^{\delta_{N}\omega(t)}dt}<\infty.
\end{align*}
\end{Prop}

\begin{proof}
  See \cite{BMT}, Corollary 2.5 and Remark after Corollary 2.6.
\end{proof}

\begin{Prop}
  \label{2.6}
  Let $\omega\in\W'$. Then for each $N\in\mathbb{N}$ and
$\varepsilon>0$ there exists $H\in\mathcal C^{\infty}(\mathbb{R}^{N})$,
$H\not=0$, with
\begin{align*}
  & {\displaystyle \mbox{\ensuremath{\supp}}(H)}
  \subset[-\varepsilon,\varepsilon]^{N}\\
  & {\displaystyle \int_{\mathbb{R}^{N}}|\hat{H}(T)|
      e^{m\omega(t)}dt}<\infty,\qquad\forall m>0.
\end{align*}
\end{Prop}

\begin{proof}
See \cite{BMT}, Corollary 2.6 and the related Remark.
\end{proof}

The difference, in the next two lemmas, between taking $\omega\in\W$ or 
$\omega\in\W'$, will be crucial in the sequel for the choice of
$\omega\in\W$ when defining the space of 
$\omega$-ultradifferentiable functions of Roumieu type and $\omega\in\W'$
for defining the space of $\omega$-ultradifferentiable functions of
Beurling type.

\begin{Lemma}
\label{3.3A}
Let $\omega\in\W$ and let $f\in{\D}(\mathbb{R}^{N})$.
If there exists $B>0$ such that
\[
{\displaystyle \int_{\mathbb{R}^{N}}|\hat{f}(t)|e^{B\omega(t)}dt:=C<\infty},
\]
then 
\begin{equation}
\underset{\alpha\in\mathbb{N}_{0}^{N}}{\sup}
\underset{x\in\mathbb{R}^{N}}{\sup}|f^{(\alpha)}(x)|
e^{-B\varphi^{*}\left(\frac{|\alpha|}{B}\right)}\leq\frac{C}{(2\pi)^{N}}.
\label{(a)}
\end{equation}
If (\ref{(a)}) holds for $f\in{\D}(\R^N)$ and $B>0$ then there
is $D>0$, depending only on $\omega,\,N\,\mbox{and}\,\,B$, and there
is $L>0$ depending only on $\omega\,\mbox{and}\,\,N$, such that
for $K=\supp f$ and $m_{N}(K)$ its Lebesgue
measure, we have that 
\[
|\hat{f}(z)|\leq m_{N}(K)\frac{CD}{(2\pi)^{N}}
e^{\left(H_{K}(\Im z)-\frac{B}{L}\omega(z)\right)}\,\,\,\,\,\,\,\,\,\,
\forall z\in\mathbb{C}^{N}.
\]
\end{Lemma}

\begin{proof}
See \cite{BMT}, Lemma 3.3.
\end{proof}

\begin{Lemma}
\label{3.3 B}
Let $\omega\in\W'$ and $f\in{\D}(\mathbb{R}^{N})$.
If there is $B>0$ such that
\[
{\displaystyle \int_{\mathbb{R}^{N}}|\hat{f}(t)|e^{B\omega(t)}dt:=C<\infty},
\]
then
\begin{equation}
\underset{\alpha\in\mathbb{N}_{0}^{N}}{\sup}
\underset{x\in\mathbb{R}^{N}}{\sup}|f^{(\alpha)}(x)|
e^{-B\varphi^{*}\left(\frac{|\alpha|}{B}\right)}\leq
\frac{C}{(2\pi)^{N}}.
\label{(a)-1}
\end{equation}
If (\ref{(a)-1}) holds for $f\in{\D}(\R^N)$ and $B>0$ then there
is $D>0$, depending only on $\omega,\,N\,\mbox{and}\,\,B$, and there
is $L>0$ depending only on $\omega\,\mbox{and\,}\,N$, such that
for $K=\supp f$ and $m_{N}(K)$ its Lebesgue measure, we have that
\begin{equation}
|\hat{f}(z)|\leq m_{N}(K){\displaystyle \frac{CD}{(2\pi)^{N}}
e^{\left(H_{K}(\Im z)+\left(\frac{1}{b}-\frac{B}{L}\right)\omega(z)\right)}
\,\,\,\,\,\,\,\,\,\,\forall z\in\mathbb{C}^{N},}
\label{(b')}
\end{equation}
where $b>0$ is the constant of condition $(\gamma)'$ in
Definition~\ref{def:2.1.2}.
\end{Lemma}

\begin{proof}
  The proof of (\ref{(a)-1}) is the same of that of (\ref{(a)}) in
  Lemma~\ref{3.3A} (see \cite[Lemma~3.3]{BMT}).
So we prove (\ref{(b')}).

By condition $(\alpha)$ there is $L>0$ such that 
\begin{equation}
\omega(Nr)\leq L\omega(r)+L\,\,\,\,\,\,\,\,\,\,\forall r>0.
\label{eq:*}
\end{equation}
Let now $z\in\mathbb{C}^{N}$ be given, let $l$ be the index with
\[
|z_{l}|=\underset{1\leq j\leq N}{\max}|z_{j}|
\]
and assume $|z_{l}|>1$.
Write then
\begin{align*}
\hat{f}(z) =\int_{K}f(t)e^{-i<t,z>}dt
 =\int_{K}\bigg(\frac{\partial^{j}}{\partial t_{l}^{j}}f(t)\bigg)
\cdot\frac{1}{(iz_{l})^{j}}\cdot e^{-i<t,z>}dt
\end{align*}
by partial integration, for all $j\in\mathbb{N}_{0}:=\N\cup\{0\}$.

In view of (\ref{(a)-1}), this implies that, for all $j\in\mathbb{N}_{0}$:
\begin{align}
|\hat{f}(z)| & \leq m_{N}(K)\frac{C}{(2\pi)^{N}}
e^{\left(B\varphi^{*}\left(\frac{j}{B}\right)-j\log|z_{l}|+H_{K}(\Im z)\right)}.
\label{c}
\end{align}

Now, note that for every $x>0$ there exists $j\in\N_{0}$ such that
$j\leq Bx<j+1$, and hence from (\ref{eq:*}) and $(\gamma)'$
\begin{align}
\underset{j\in\mathbb{N}_{0}}{\sup}\left(j\log|z_{l}|-B
\varphi^{*}\left(\frac{j}{B}\right)\right) & =
B\underset{j\in\mathbb{N}_{0}}{\sup}\left(\frac{j+1}{B}
\log|z_{l}|-\varphi^{*}\left(\frac{j}{B}\right)\right)-\log|z_{l}|\nonumber \\
 & \geq B\,\underset{x>0}{\sup}\left(x\log|z_{l}|-\varphi^{*}(x)\right)-
\log|z_{l}|\nonumber \\
 & =B\varphi^{**}\left(\log|z_{l}|\right)-\log|z_{l}|\nonumber \\
 & =B\omega(z_{l})-\log|z_{l}|\nonumber \\
 & \geq B\omega\left(\frac{z}{N}\right)-\log|z|\nonumber \\
 & \geq\frac{B}{L}\omega(z)-1-\log|z|\nonumber \\
 & \geq\frac{B}{L}\omega(z)-1-\frac{\omega(z)}{b}+\frac{a}{b}\nonumber \\
 & =\left(\frac{B}{L}-\frac{1}{b}\right)\omega(z)+\left(\frac{a}{b}-1\right).
\label{*}
\end{align}

By passing to the infimum over all $j\in\mathbb{N}_{0}$ in (\ref{c})
and by using (\ref{*}) we obtain:
\begin{align*}
|\hat{f}(z)| & \leq{\displaystyle m_{N}(K)\frac{C}{(2\pi)^{N}}
e^{\left\{ \left(\frac{1}{b}-\frac{B}{L}\right)\omega(z)+\left(1-\frac{a}{b}\right)
+H_{K}(\Im z)\right\} }}\\
 & =m_{N}(K)\frac{CD}{(2\pi)^{N}}e^{\left\{ \left(\frac{1}{b}-\frac{B}{L}\right)
\omega(z)+H_{K}(\Im z)\right\} },
\end{align*}
where $D=e^{\left(1-\frac{a}{b}\right)}$.
\end{proof}

\begin{Def}
\begin{em}
Let $\omega\in\W$ and let $K\subset\mathbb{R}^{N}$ be a compact
set.
For $\lambda>0$ we define the Banach space
\begin{equation}
{\D}_{\lambda}(K)=\left\{ f\in\mathcal C^{\infty}(\mathbb{R}^{N})|\,
\supp f\subset K\,\,\,\mbox{and}\,\,\,\|f\|_{\lambda}:=
\int_{\mathbb{R}^{N}}|\hat{f}(t)|e^{\lambda\omega(t)}dt<\infty\right\} .
\label{eq:iaia}
\end{equation}
We set
\[
\D_{\{\omega\}}(K)=\indlim_{\lambda\to0}\D_{\lambda}(K)
\]
endowed with the topology of the inductive limit.

For an open set $\Omega\subset\mathbb{R}^{N}$ we define then
\[
\D_{\{\omega\}}(\Omega)=\indlim_{K\subset\!\subset\Omega}\D_{\{\omega\}}(K)
\]
where the inductive limit is taken over all compact subsets $K$ of
$\Omega$. We endow $\D_{\{\omega\}}(\Omega)$ with the inductive limit topology.

The elements of $\D_{\{\omega\}}(\Omega)$ are called 
\emph{$\omega-$ultradifferentiable functions of Roumieu type
with compact support.}
\end{em}
\end{Def}

\begin{Def}
\begin{em}
Let $\omega\in\W'$ and let $K\subset\mathbb{R}^{N}$ be a compact
set.

For ${\D}_{\lambda}(K)$ defined as in (\ref{eq:iaia}), we set
\[
\D_{(\omega)}(K)=\projlim_{\lambda\rightarrow\infty}\D_{\lambda}(K)
\]
endowed with the topology of the projective limit.

For an open set $\Omega\subset\mathbb{R}^{N}$ we define
\[
\D_{(\omega)}(\Omega)=\indlim_{K\subset\!\subset\Omega}\D_{(\omega)}(K)
\]
where the inductive limit is taken over all compact subsets of $\Omega$.
We endow
$\D_{(\omega)}(\Omega)$ with the inductive limit topology.

The elements of $\D_{(\omega)}(\Omega)$ are called
\emph{$\omega-$ultradifferentiable functions of Beurling
type with compact support.}
\end{em}
\end{Def}

\begin{Rem}
  \label{rem218}
\begin{em}
As in \cite{BMT}, we have the following:
\begin{enumerate}
\item
  Let $K\subset\mathbb{R}^{N}$ with non-empty interior.
  If $\omega\in\W'$ then $\D_{(\omega)}(K)\neq\{0\}$;
  if $\omega\in\W$ then $\D_{\{\omega\}}(K)\neq\{0\}$ and moreover
  $\D_{(\omega)}(K)\subset\D_{\{\omega\}}(K)$.
\item
For $\omega,\sigma\in\W'$ we have that 
$\D_{(\omega)}(\mathbb{R})\subset\D_{(\sigma)}(\mathbb{R})$
iff $\sigma=O(\omega)$.
\item
We say that two functions $\omega$ and $\sigma$ are equivalent if
$\omega=O(\sigma)$ and $\sigma=O(\omega)$. Note that if 
$\omega\leq\sigma\leq C\omega$
for some $C>0$ and if $\psi(x)=\sigma(e^{x})$, then
\[
C\varphi^{*}\left(\frac{y}{C}\right)\leq\psi^{*}(y)\leq
\varphi^{*}(y)\,\,\,\,\,\,\,\,\,\,\forall y>0.
\]
With this formula, it's easy to see that definitions and most theorems
in the sequel don't change if $\omega$ is only equivalent to a weight
function.
\end{enumerate}
\end{em}
\end{Rem}

Lemmas~\ref{3.3A} and \ref{3.3 B} and the classical Paley-Wiener Theorem
for $\D(K)$ imply the following Paley-Wiener theorems for $\omega-$
ultradifferentiable functions:
\begin{Th}[Paley-Wiener Theorem for $\omega$-ultradifferentiable
functions of Roumieu type]
\label{prop:2.3.5}
\ 
Let
$\omega\in\W$ , $K\subset\mathbb{R}^{N}$ a convex compact set and
$f\in L^{1}(\mathbb{R}^{N})$.
The following are equivalent:

(1) $f\in\D_{\{\omega\}}(K)$, 

(2) $f\in\D(K)$ and for some $k\in\mathbb{N}$
\[
\underset{\alpha\in\mathbb{N}_{0}^{N}}{\sup}
\underset{x\in\mathbb{R}^{N}}{\sup}|f^{(\alpha)}(x)|
e^{\left(-\frac{1}{k}\varphi^{*}(|\alpha|k)\right)}<\infty,
\]

(3) there exist $\varepsilon,\,C>0$ such that 
\[
|\hat{f}(z)|\leq Ce^{\left(H_{K}(\Im z)-\varepsilon\omega(z)\right)}
\,\,\,\,\,\,\,\,\,\,\forall z\in\mathbb{C}^{N}.
\]
\end{Th}

\begin{proof}
See \cite[Prop. 3.4]{BMT}, or \cite{G} for more details.
\end{proof}

\begin{Th}[Paley-Wiener Theorem for $\omega$-ultradifferentiable
functions of Beurling type]
\label{prop:2.3.6}
\ 
Let $\omega\in\W'$, $K$ $\subset\mathbb{R}^{N}$ a convex compact set
and $f\in L^{1}(\mathbb{R}^{N})$. 
The following are equivalent:

(1) $f\in\D_{(\omega)}(K)$,

(2) $f\in\D(K)$ and for all $k\in\mathbb{N}$
\[
\underset{\alpha\in\mathbb{N}_{0}^{N}}{\sup}
\underset{x\in\mathbb{R}^{N}}{\sup}|f^{(\alpha)}(x)|
e^{\left(-k\varphi^{*}(\frac{|\alpha|}{k})\right)}<\infty,
\]

(3) for all $k\in\mathbb{N}$ there is $C_{k}>0$ such that
\[
|\hat{f}(z)|\leq C_{k}e^{\left(H_{K}(\Im z)-k\omega(z)\right)}
\,\,\,\,\,\,\,\,\,\,\forall z\in\mathbb{C}^{N}.
\]
\end{Th}

\begin{proof}
  $\mathbf{(1)\Rightarrow(2)}:$
  
If $f\in\D_{(\omega)}(K)$ then, by definition, $f\in\D(K)$ and for every
$\varepsilon>0$ 
\[
\int_{\mathbb{R}^{N}}|\hat{f}(t)|e^{\varepsilon\omega(t)}dt=:C_{\varepsilon}<\infty.
\]
So, by Lemma \ref{3.3 B}, for all $\varepsilon>0$
\begin{equation}
\underset{\alpha\in\mathbb{N}_{0}^{N}}{\sup}
\underset{x\in\mathbb{R}^{N}}{\sup}|f^{(\alpha)}(x)|
e^{\left(-\varepsilon\varphi^{*}\left(\frac{|\alpha|}{\varepsilon}\right)\right)}
<\infty,
\label{eq:starstar}
\end{equation}
and hence (2), since ${\displaystyle \frac{\varphi^{*}(x)}{x}}$ is increasing.

\textbf{$\mathbf{(2)\Rightarrow(3)}$}:

If $f\in\D(K)$ satisfies (2) then it
also satisfies (\ref{eq:starstar}) for every $\varepsilon>0$ since 
$\varphi^*(s)/s$ is increasing, and hence,
by Lemma \ref{3.3 B}, there exists $D_{\varepsilon}>0$ such that 
\begin{equation}
|\hat{f}(z)|\leq{\displaystyle D_{\varepsilon}
e^{\left(H_{K}(\Im z)+\left(\frac{1}{b}-\frac{\varepsilon}{L}\right)\omega(z)\right)}.}
\label{eq:stella}
\end{equation}
Therefore for every $\tilde{\varepsilon}>0$ we can choose 
\[
\varepsilon=L\left(\tilde{\varepsilon}+\frac{1}{b}\right)>0
\]
in (\ref{eq:starstar}), so that (\ref{eq:stella}) becomes 
\[
|\hat{f}(z)|\leq{\displaystyle D_{\varepsilon}
e^{(H_{K}(\Im z)-\tilde{\varepsilon}\omega(z))},}
\]
and hence (3).

$\mathbf{(3)\Rightarrow(1)}:$

By (3) and $(\gamma)'$ we have that for all $\lambda>0$, taking
$k\in\N$ with $k>\lambda$, there exist $C_{\lambda},\,C'_{\lambda}>0$
such that
\begin{align}
\nonumber
\int_{\R^N}|\hat{f}(t)|e^{\lambda\omega(t)}dt & 
\leq C_{\lambda}\int_{\R^N}e^{(-k+\lambda)\omega(t)}dt\\
\label{G2}
 & \leq C_{\lambda}\int_{\R^N}e^{(-k+\lambda)\left(a+b\log(1+t)\right)}dt\\
\nonumber
 & =C'_{\lambda}\int_{\R^N}(1+t)^{b(\lambda-k)}dt.
\end{align}
For ${\displaystyle k>\frac{N+1}{b}+\lambda}$ the above integral
is finite and hence there exists $C''_{\lambda}>0$ such that 
\[
\int_{\R^N}|\hat{f}(t)|e^{\lambda\omega(t)}dt\leq C''_{\lambda}.
\]
To prove that $f\in{\D}(K)$ note that (3) and $(\gamma)'$ imply
that for every $k\in\N$ there exists $C_{k}>0$ such that 
\begin{align*}
\left|\hat{f}(z)\right| & \leq C_{k}e^{H_{K}(\Im z)-k\omega(z)}\\
 & \leq C_{k}e^{H_{K}(\Im z)-k\left(a+b\log(1+|z|)\right)}\\
 & =C_{k}e^{-ak}e^{H_{K}(\Im z)}\left(1+|z|\right)^{-bk}
\,\,\,\,\,\,\,\,\,\,\forall z\in\mathbb{C}^{N}.
\end{align*}
Therefore for every $n\in\N$ there exists $C_{n}>0$ such that
\[
\left|\hat{f}(z)\right|\leq C_{n}
e^{H_{K}(\Im z)}\left(1+|z|\right)^{-n}
\,\,\,\,\,\,\,\,\,\,\forall z\in\mathbb{C}^{N}.
\]
By the classical Paley-Wiener Theorem  we finally
have that $f\in{\D}(K)$ and hence the theorem is proved. 
\end{proof}

\begin{Rem}
\begin{em}
The inequality \eqref{G2} enlightens the sufficiency of condition $(\gamma)'$
on the weight $\omega$: by the arbitrariety of $\lambda$ we can allow a fixed 
$b>0$ to make the integral convergent. On the contrary, in the Roumieu
case (Theorem~\ref{prop:2.3.5}) we need condition $(\gamma)$, i.e.
$\log(1+t)=o(\omega(t))$ as 
$t\to+\infty$, since $\lambda$ is fixed.
\end{em}
\end{Rem}

For a sequence $\mathbb{P}=(p_{n})_{n\in\N}$ of continuous functions
$p_{n}:\mathbb{C}^{N}\rightarrow\mathbb{R}$, we define
\[
A_{\mathbb{P}}(\mathbb{C}^{N}):=\left\{ f\in{\mathcal O}(\mathbb{C}^{N})\,|\,
\mbox{\,for some\,\,}n:\,\,\underset{z\in\mathbb{C}^{N}}{\sup}
|f(z)|e^{-p_{n}(z)}<\infty\right\} ,
\]
and
\[
A_{\mathbb{P}}^{0}(\mathbb{C}^{N}):=\left\{ f\in{\mathcal O}(\mathbb{C}^{N})\,|\,
\mbox{\,for all\,\,}n:\,\,\underset{z\in\mathbb{C}^{N}}{\sup}|f(z)|
e^{-p_{n}(z)}<\infty\right\} ,
\]
where ${\mathcal O}(\mathbb{C}^{N})$ is the set of all entire functions
on $\mathbb{C}^{N}$.

Let $\omega$ be a weight function and $K\subset\mathbb{R}^{N}$ a
convex compact set. Define
\[
\mathbb{P}_K:=\left\{ p_{n}:z\mapsto H_{K}(\Im z)-\frac{1}{n}\omega(z),\,n
\in\mathbb{N}\right\} 
\]
and
\[
\mathbb{M}_K:=\left\{ m_{n}:z\mapsto H_{K}(\Im z)-n\omega(z),\,n\in
\mathbb{N}\right\} .
\]

For an open convex set $\Omega\subset\R^N$ and a convex compact exhaustion
$\mathring{K_{1}}\subset\mathring{K_{2}}\subset
\mathring{K_{3}}\subset\ldots$ of $\Omega$, define also
\beqsn
\mathbb{P}_\Omega:=\left\{ p_{n}:z\mapsto H_{K_n}(\Im z)-\frac{1}{n}\omega(z),\,n
\in\mathbb{N}\right\}.
\eeqsn

From the Paley-Wiener Theorems \ref{prop:2.3.5}-\ref{prop:2.3.6}
(cf. also \cite[Prop. 3.5]{BMT}) we get:

\begin{Prop}
We have the following:
\begin{enumerate}
\item
Let $K$ be a compact convex set of $\R^N$. Then:
\begin{enumerate}
\item
if $\omega\in\W$ 
\[
\D_{\{\omega\}}(K)\cong A_{\mathbb{P}_K}(\mathbb{C}^{N});
\]
\item
if $\omega\in\W'$
\[
\D_{(\omega)}(K)\cong A_{\mathbb{M}_K}^{0}(\mathbb{C}^{N}).
\]
\end{enumerate}
\item
For $\omega\in\W$, $\Omega\subset\mathbb{R}^{N}$ an open
convex set:
\[
\D_{\{\omega\}}(\Omega)\cong A_{\mathbb{P}_{\Omega}}(\mathbb{C}^{N}).
\]
\end{enumerate}
The isomorphisms are given by the Fourier-Laplace transform.
\end{Prop}

As in \cite{BMT}, we can collect some more properties on these spaces of 
$\omega$-ultradifferentiable functions with compact support, taking 
$\omega\in\W'$ in the Beurling case and $\omega\in\W$ in the Roumieu case.

\begin{Cor}
  Let $K\subset\mathbb{R}^{N}$ be compact and $\Omega\subset\mathbb{R}^{N}$ be
  open.

(1) Let $\omega\in\W$. Then $\D_{\{\omega\}}(K)$ is a (DFN)-space, i.e. the strong
dual of a nuclear Fr\'echet space. In particular, it's complete, reflexive
and nuclear.

(2) Let $\omega\in\W'$. Then $\D_{(\omega)}(K)$ is a (FN)-space, i.e. a nuclear
Fr\'echet space.
\end{Cor}

\begin{Lemma}
Let $\omega\in\W$, $f\in\D(\R^N)$, $g\in\D_{\{\omega\}}(\R^N)$.
Then we have:

$\ds{\it (1)_{\{\omega\}}}$
$\quad f*g\in\D_{\{\omega\}}(\R^N)$,

$\ds{\it (2)_{\{\omega\}}}$
$\quad\supp\, (f*g)\subset\supp f+\supp g$,

$\ds{\it (3)_{\{\omega\}}}$
$\quad\widehat{f*g}(z)=\hat{f}(z)\hat{g}(z)$.

\vspace*{1.5mm}
\noindent
Let $\omega\in\W'$, $f\in\D(\R^N)$, $g\in\D_{(\omega)}(\R^N)$.
Then we have:

\vspace*{1mm}
$\ds{\it (1)_{(\omega)}}$ $\quad f*g\in\D_{(\omega)}(\R^N)$,

$\ds{\it (2)_{(\omega)}}$ $\quad\supp\,(f*g)\subset\supp f+\supp g$,

$\ds{\it (3)_{(\omega)}}$ $\quad\widehat{f*g}(z)=\hat{f}(z)\hat{g}(z)$.
\end{Lemma}

\begin{Lemma}
\label{lem:3.8}
Let $K_{1},\,K_{2}\subset\R^N$ be compact sets with 
$K_{1}\subset\mathring{K_{2}}$.

(a) Let $\omega,\,\sigma\in\W$ with $\sigma\leq\omega$. Then
for all $f\in{\D}_{\{\sigma\}}(K_{1})$ there is a sequence 
$\{f_{n}\}_{n\in\mathbb{N}}$
in $\D_{\{\omega\}}(K_{2})$ with $\underset{n\rightarrow\infty}{\lim}f_{n}=f$
in ${\D}_{\{\sigma\}}(K_{2})$.\\

(b) Let $\omega,\,\sigma\in\W'$ with $\sigma\leq\omega$. Then
for all $f\in{\D}_{(\sigma)}(K_{1})$ there is a sequence 
$\{f_{n}\}_{n\in\mathbb{N}}$
in $\D_{(\omega)}(K_{2})$ with $\underset{n\rightarrow\infty}{\lim}f_{n}=f$
in ${\D}_{(\sigma)}(K_{2})$.\\

(c) Let $\omega\in\W'$. Then for all $f\in{\D}(K_{1})$ there
is a sequence $\{f_{n}\}_{n\in\mathbb{N}}$ in $\D_{(\omega)}(K_{2})$
with $\underset{n\rightarrow\infty}{\lim}f_{n}=f$ in $\D(K_{2})$.
\end{Lemma}

\begin{Prop}
\label{prop:3.9}
Let $\omega,\,\sigma\in\W$ with $\sigma=o(\omega)$.
Then the inclusions 
\[
\D_{(\omega)}(\Omega)\hookrightarrow\D_{\{\omega\}}(\Omega)
\hookrightarrow\D_{(\sigma)}(\Omega)\hookrightarrow\D(\Omega)
\]
are continuous and sequentially dense for each open set $\Omega\subset\R^N$.
\end{Prop}

Let us now introduce the algebras of $\omega$-ultradifferentiable
functions of Beurling and of Roumieu type with arbitrary support.
Here again we need $\omega\in\W$ in the Roumieu case, while we can allow
$\omega\in\W'$ in the Beurling case.

\begin{Def}
  \label{def227}
\begin{em}
For
$\omega\in\W$ and an open set $\Omega\subset\R^N$, we define
\begin{align*}
  \E_{\{\omega\}}(\Omega):= & \Big\{f\in{\mathcal C}^\infty(\Omega)|
  \mbox{\,\,for all compact}\,\,
K\subset\Omega\,\mbox{\,there is }m\in\mathbb{N}\,\mbox{\,with}\\
 & \underset{\alpha\in\mathbb{N}_{0}^{N}}{\sup}\underset{x\in K}{\sup}
|f^{(\alpha)}(x)|e^{\left(-\frac{1}{m}\varphi^{*}(m|\alpha|)\right)}<\infty\Big\}.
\end{align*}

For $\omega\in\W'$ and an open set $\Omega\subset\R^N$ we define
\begin{align*}
  \E_{(\omega)}(\Omega):= & \Big\{f\in{\mathcal C}^\infty(\Omega)|
  \mbox{\,\,for all compact}\,\,
K\subset\Omega\,\,\mbox{and all}\,\,m\in\mathbb{N}\,\\
 & p_{K,m}(f):=\underset{\alpha\in\mathbb{N}_{0}^{N}}{\sup}
\underset{x\in K}{\sup}|f^{(\alpha)}(x)|
e^{\left(-m\varphi^{*}\left(\frac{|\alpha|}{m}\right)\right)}<\infty\Big\}.
\end{align*}

The topology of $\E_{\{\omega\}}(\Omega)$ is given by first taking the inductive
limit over all $m\in\mathbb{N}$ for each compact $K\subset\Omega$
and then taking the projective limit for $K\subset\Omega$, while
$\E_{(\omega)}(\Omega)$ carries the metric locally convex topology given by the
seminorms $p_{K,m}$ where $K$ is a compact subset of $\Omega$ and
$m\in\mathbb{N}$.

The elements of $\E_{\{\omega\}}(\Omega)$ are called 
\emph{$\omega-$ultradifferentiable
  functions of Roumieu type}, while the elements of $\E_{(\omega)}(\Omega)$
are called
\emph{$\omega-$ultradifferentiable functions of Beurling type}.
\end{em}
\end{Def}

\noindent{\bf Notation.}
We shall write $\E_*$ (resp. $\D_*$) if a statements holds for both
$\E_{\{\omega\}}$ (resp. $\D_{\{\omega\}})$ and $\E_{(\omega)}$ 
(resp. $\D_{(\omega)}$), taking $\omega\in\W$
in the Roumieu case and $\omega\in\W'$ in the Beurling case.

\begin{Ex}
\begin{em}
For $\omega$ as in (\ref{eq:(1)-1}) the space $\E_{\{\omega\}}(\Omega)$ is
the classical Gevrey class of order ${\displaystyle \frac{1}{\alpha}.}$
For $\omega$ as in \eqref{212} with $\beta=1$, the
space $\E_{(\omega)}(\Omega)$ is the space $\E(\Omega)$ of $\mathcal C^\infty$
functions in $\Omega$.
\end{em}
\end{Ex}

\begin{Rem}
  \begin{em}
    In general the spaces of $\omega$-ultradifferentiable functions
    defined as in Definition~\ref{def227} are different from the
    Denjoy-Carleman classes of ultradifferentiable functions as defined
    in \cite{K} (cf. \cite{BMM}).
      \end{em}
  \end{Rem}

As in \cite{BMT}, we have the following properties of the spaces
$\E_*(\Omega)$:
\begin{Prop}
$\E_*(\Omega)$ is a locally convex algebra with continuous moltiplication.
\end{Prop}

\begin{Lemma}
Let $\Omega\subset\R^N$ be open, let 
$K_{1}\subset\mathring{K_{2}}\subset K_{2}\subset\ldots\subset\Omega$
be an exhaustion of $\Omega$ by compact sets, choose $\chi_{j}\in\D_*(K_{j})$
with $0\leq\chi_{j}\leq1$ and $\left.\chi\right|_{K_{j-1}}\equiv1$.
We thus have maps
\begin{align*}
\D_*(K_{j+1}) & \rightarrow\D_*(K_{j})\\
f & \mapsto\chi_{j}f
\end{align*}
by which
\[
\E_*(\Omega)=\underset{j\rightarrow\infty}{\projlim}\D_*(K_{j}).
\]
\end{Lemma}

\begin{Lemma}
Let $K$ be a compact subset of an open set $\Omega\subset\R^N$. Then
$\D_*(K)$ carries the topology which is induced by $\E_*(\Omega)$.
\end{Lemma}

\begin{Prop}
  The following properties hold:
  
  (1) The inclusion $\D_*(\Omega)\hookrightarrow\E_*(\Omega)$ is continuous
and has dense image.

(2) Let $\omega,\,\sigma\in\W$ with $\sigma=o(\omega)$, then the
inclusion $\E_{\{\omega\}}(\Omega)\hookrightarrow{\E}_{(\sigma)}(\Omega)$
is continuous and has dense image.
\end{Prop}

\begin{Prop}
  Let $\Omega\subset\R^N$ be open and let
  $\left\{\Omega_{j}\right\}_{j\in\N}$ be an open covering of $\Omega$.
Then there are $f_{j}\in\D_*(\Omega_{j})$ with $0\leq f_{j}\leq1$
such that $\sum_{j=1}^{\infty}f_{j}=1$ and $\left\{\supp f_{j}\right\}_{j\in\N}$
is locally finite.
\end{Prop}

\begin{Prop}
Let $\Omega_{1}$, $\Omega_{2}$ be given open subsets of $\R^N$, let 
$g:\Omega_{1}\rightarrow\Omega_{2}$
be real-analytic, and let $f\in\E_*(\Omega_{2})$. Then 
$f\circ g\in\E_*(\Omega_{1})$.
In particular, $\E_*(\Omega)$ contains all real-analytic functions
on $\Omega$.
\end{Prop}

Let us now introduce the $\omega-$ultradistributions of Beurling
and of Roumieu type with compact and arbitrary support, taking
$\omega\in\W$ in the Roumieu case and $\omega\in\W'$ in the Beurling case.

\begin{Def}
\begin{em}
Let $\omega\in\W$
and $\Omega\subset\R^N$ an open set.

(1) The elements of $\D_{\{\omega\}}'(\Omega)$ are called 
\emph{$\omega-$ultradistributions
of Roumieu type}.

(2) For an ultradistribution $T\in\D_{\{\omega\}}'(\Omega)$ its support 
$\supp T$
is the set of all points such that for every neighbourhood $U$ there
is $\varphi\in\D_{\{\omega\}}(U)$ with $<T,\varphi>\not=0$.
\end{em}
\end{Def}

\begin{Def}
\begin{em}
Let $\omega\in\W'$
and $\Omega\subset\R^N$ an open set.

(1) The elements of $\D_{(\omega)}'(\Omega)$ are called 
\emph{$\omega-$ultradistributions
of Beurling type.}

(2) For an ultradistribution $T\in\D_{(\omega)}'(\Omega)$ its support 
$\supp T$
is the set of all points such that for every neighbourhood $U$ there
is $\varphi\in\D_{(\omega)}(U)$ with $<T,\varphi>\not=0$.
\end{em}
\end{Def}

\begin{Rem}
\begin{em}
By Proposition \ref{prop:3.9}, the definition of support of an ultradistribution
$T$ doesn't depend on the choice of the class $\D_{\{\omega\}}(\Omega)$ for
$\omega\in\W$ (resp. $\D_{(\omega)}(\Omega)$ for $\omega\in\W')$ as long
as it contains $T$. In particular, if $T$ is a distribution 
$T\in{\D}'(\Omega)$,
then the support defined above is the usual one.
\end{em}
\end{Rem}

As in \cite[Prop. 5.3]{BMT}, the elements of $\E'_*(\Omega)$ can be identified
with distributions in $\D'_*(\Omega)$ with compact support:
\begin{Prop}
\label{prop:5.3}
An ultradistribution $T\in\D_*'(\Omega)$ can be extended
continuously to $\E_*(\Omega)$ iff $\supp T$ is
a compact subset of $\Omega$.
\end{Prop}

\begin{Def}
\begin{em}
  Let $\Omega\subset\R^N$ be open. For $f\in\E_*(\Omega)$ and
  $T\in\D_*'(\Omega)$
we define $fT\in\D_*'(\Omega)$ by 
\[
\left\langle fT,\varphi\right\rangle =\left\langle T,f\varphi
\right\rangle \,\,\,\,\,\,\,\,\,\,\forall\varphi\in\D_*(\Omega).
\]
This makes $\D_*'(\Omega)$ an $\E_*(\Omega)-$module.
\end{em}
\end{Def}

\begin{Def}
  \begin{em}
For an ultradistribution $\mu\in\E_*'(\R^N)$, and for $f\in\E_*(\R^N)$ we
define the convolution 
\begin{align*}
T_{\mu}(f): & =\mu*f:\R^N\rightarrow\mathbb{C},
\end{align*}
by
\[
\mu*f(x)=\left\langle \mu_{y},f(x-y)\right\rangle .
\]
\end{em}
\end{Def}

As in \cite[Prop. 6.3]{BMT}:
\begin{Prop}
The convolution map
\[
T_{\mu}:\E_*(\R^N)\rightarrow\E_*(\R^N)
\]
is continuous.
  \end{Prop}

\noindent{\bf Notation.}
For
$z\in\mathbb{C}^{N}$ we set
\[
f_{z}(x)=e^{-i\left\langle x,z\right\rangle },\,\,\,\,\,\,\,\,\,\,x\in\R^N.
\]

For each $\lambda>0$ we have 
\begin{align}
 \sup_{\alpha\in\N_0^N}\sup_{x\in\R^N}
  |f_{z}^{(\alpha)}(x)|e^{-\lambda\varphi^*\left(\frac{|\alpha|}{\lambda}\right)}
  & =\sup_{\alpha\in\N_0^N}\sup_{x\in\R^N}
  |z^{\alpha}|
  |e^{-i\left\langle x,z\right\rangle }|e^{-\lambda\varphi^*\left(\frac{|\alpha|}{\lambda}\right)}
  \nonumber \\
  & \leq\sup_{x\in\R^N}e^{\langle
      x,\Im z\rangle} \cdot\exp\left\{\sup_{\alpha\in\N_{0}^{N}}
      \left(|\alpha|\log|z|-\lambda\varphi^{*}\left(\frac{|\alpha|}{\lambda}
      \right)\right)\right\}\nonumber \\
  & =e^{H_{K}(\Im z)}\exp\left\{ \sup_{\alpha\in\N_{0}^{N}}\lambda
  \left(\frac{|\alpha|}{\lambda}\log|z|-\varphi^{*}\left(
  \frac{|\alpha|}{\lambda}\right)\right)\right\} \nonumber \\
 & =e^{H_{K}(\Im z)}\exp\left\{ \lambda\varphi^{**}(\log|z|)\right\} \nonumber \\
   & =e^{H_{K}(\Im z)+\lambda\omega(z)}.
  \label{eq:AA}
\end{align}
Thus $f_{z}\in\E_*(\Omega)$ for all $\omega$ and $\Omega$.

\begin{Def}
  \begin{em}
The {\em Fourier-Laplace transform} $\hat{\mu}$ of $\mu\in\E_*'(\Omega)$
is defined by
\[
\hat{\mu}:z\mapsto\left\langle \mu,f_{z}\right\rangle .
\]
  \end{em}
\end{Def}

Note that for $\varphi\in\D_*(\Omega):$
\begin{align*}
\widehat{\mu*\varphi}(z) & =\int_{\R^N}\mu*\varphi(t)f_{z}(t)dt
  =\int_{\R^N}\left\langle \mu_{y},f_{z}(t)\varphi(t-y)\right\rangle dt\\
 & =\int_{\R^N}\left\langle \mu_{y},f_{z}(s+y)\varphi(s)\right\rangle ds
  =\int_{\R^N}\left\langle \mu,f_{z}\right\rangle \varphi(s)f_{z}(s)ds\\
 & =\left\langle \mu,f_{z}\right\rangle \int_{\R^N}\varphi(s)f_{z}(s)ds
  =\hat{\mu}(z)\hat{\varphi}(z).
\end{align*}

\begin{Th}[Paley-Wiener theorem for $\omega$-ultradistributions of
    Beurling type]
  \label{thm:prop 7.2-7.3}
  Let $\omega\in\W'$ and $K\subset\R^N$ compact and convex. If
  $\mu\in\E_{(\omega)}'(\R^N)$
with $\supp \mu\subset K$ then $\hat{\mu}$ is entire and there
exist $C,\lambda>0$ such that
\begin{equation}
  \left|\hat{\mu}(z)\right|\leq Ce^{H_{K}(\Im z)+\lambda\omega(z)}
  \,\,\,\,\,\,\,\,\,\,\forall z\in\mathbb{C}^{N}.
  \label{eq:(2)-1}
\end{equation}
This holds, in particular, for $K$ equal to the convex hull of $\supp\mu$.
Moreover,
\begin{equation}
  \left\langle \mu,\varphi\right\rangle =\frac{1}{(2\pi)^{N}}\int_{\R^N}
  \hat{\mu}(-t)\hat{\varphi}(t)dt
  \,\,\,\,\,\,\,\,\,\,\forall\varphi\in\D_{(\omega)}(\R^N).
  \label{eq:(3)-1}
\end{equation}
Conversely, if $g$ is an entire function on $\mathbb{C}^{N}$ that
satisfies (\ref{eq:(2)-1}), i.e.
\[
|g(z)|\leq Ce^{H_{K}(\Im z)+\lambda\omega(z)}
\,\,\,\,\,\,\,\,\,\,\,\,\,\forall z\in\mathbb{C}^{N},
\]
for some $C,\,\lambda>0$, then there exists $\mu\in\E_{(\omega)}'(\R^N)$ such
that $\hat{\mu}=g$ and $\supp\mu\subset K$.
\end{Th}

\begin{proof}
Let us first prove that if $f\in\E_{(\omega)}(\R^N)$ with 
$\left.f\right|_{K}\equiv0$
then $\left\langle \mu,f\right\rangle =0$.

To this aim we assume, without loss of generality, that $0\in\mathring{K}$
and define
\[
f_{t}(x):=f(tx),\,\,\,\,\,\,\,\,\,\,0<t<1.
\]
Then $\left\langle \mu,f_{t}\right\rangle =0$. Let us to prove that
\begin{align}
  {\displaystyle \underset{t\rightarrow1^{-}}{\lim}}\left\langle
  \mu,f_{t}\right\rangle  & =\left\langle \mu,f\right\rangle .
  \label{eq:(4)}
\end{align}

We have that $\mu\in\E_{(\omega)}'(\R^N)$, so $\mu$ is a linear and continuous
function on $\E_{(\omega)}(\R^N)$ and to prove (\ref{eq:(4)}) it's sufficient
to prove that $f_{t}\rightarrow f$ in $\E_{(\omega)}(\R^N)$. Therefore, fix
$\tilde{K}\subset\R^N$ compact, $m\in\N$ and prove that 
\begin{equation}
  \sup_{\alpha\in\N_0^N}\sup_{x\in\tilde K}
  |D^{\alpha}f_{t}(x)-D^{\alpha}f(x)|e^{-m\varphi^{*}\left(\frac{|\alpha|}{m}\right)}
  \rightarrow0.
  \label{eq:(5)}
\end{equation}

Indeed,
\begin{align*}
|D^{\alpha}f_{t}(x)-D^{\alpha}f(x)| & =|D^{\alpha}f(tx)-D^{\alpha}f(x)|
  =\left|t^{\alpha}(D^{\alpha}f)(tx)-D^{\alpha}f(x)\right|\\
& \leq\left|t^{\alpha}(D^{\alpha}f)(tx)-t^{\alpha}D^{\alpha}f(x)\right|
+\left|t^{\alpha}D^{\alpha}f(x)-D^{\alpha}f(x)\right|\\
& =t^{\alpha}\left|(D^{\alpha}f)(tx)-D^{\alpha}f(x)\right|+
(1-t^{\alpha})\left|D^{\alpha}f(x)\right|,
\end{align*}
so, for $0<t<1$, we have that 
\begin{align}
  \sup_{\alpha\in\N_0^N}\sup_{x\in\tilde K}
    |D^{\alpha}f_{t}(x)-D^{\alpha}f(x)|e^{-m\varphi^{*}\left(\frac{|\alpha|}{m}\right)} &
  \leq\sup_{\alpha\in\N_0^N}\sup_{x\in\tilde K}
  |(D^{\alpha}f)(tx)-D^{\alpha}f(x)|e^{-m\varphi^{*}\left(\frac{|\alpha|}{m}\right)}+
  \label{eq:(8)}\\
  & +(1-t^{\alpha})\sup_{\alpha\in\N_0^N}\sup_{x\in\tilde K}
  |D^{\alpha}f(x)|e^{-m\varphi^{*}\left(\frac{|\alpha|}{m}\right)}.
  \nonumber 
\end{align}

We observe that $(1-t^{\alpha})\rightarrow0$ for $t\rightarrow1^{-}$
and
$\ds\sup_{\alpha\in\N_0^N}\sup_{x\in\tilde K}
|D^{\alpha}f(x)| e^{-m\varphi^{*}\left(\frac{|\alpha|}{m}\right)}=C_{\tilde{K}}<\infty$
because $f\in\E_{(\omega)}(\R^N)$.

To estimate also the first addend of (\ref{eq:(8)}) let us remark
that it's not restrictive to assume $0\in\tilde{K}$, since we can
enlarge $\tilde{K}$. Therefore, denoting by $\ch(\tilde{K})$ the
convex hull of $\tilde{K}$, by the Lagrange Theorem we have that
there exists $\xi\in{\ch}(\tilde{K})$ on the segment of extremes
$x$ and $tx$, such that 
\begin{align*}
 &\sup_{\alpha\in\N_0^N}\sup_{x\in\tilde K}
  |(D^{\alpha}f)(tx)-D^{\alpha}f(x)|e^{-m\varphi^{*}\left(\frac{|\alpha|}{m}\right)} 
  =\sup_{\alpha\in\N_0^N}\sup_{x\in\tilde K}
  |\left\langle
  \nabla(D^{\alpha}f)(\xi),tx-x\right\rangle
  |e^{-m\varphi^{*}\left(\frac{|\alpha|}{m}\right)}\\
  & \leq \sup_{\alpha\in\N_0^N}
  \left\{ \underset{\xi
    \in{\ch}(\tilde{K})}{\sup}\|\nabla D^{\alpha}f(\xi)\|
  \cdot(1-t)\cdot\underset{x\in\tilde{K}}{\sup}\|x\|\cdot
  e^{-m\varphi^{*}\left(\frac{|\alpha|}{m}\right)}\right\} \\
  & \leq C(1-t)\sup_{\alpha\in\N_0^N}
  \underset{\xi\in{\ch}(\tilde{K})}{\sup}
  \|\nabla D^{\alpha}f(\xi)\|
  e^{-m\varphi^{*}\left(\frac{|\alpha|}{m}\right)},
\end{align*}
for some $C>0$.

However,
\begin{align*}
  \|\nabla D^{\alpha}f(\xi)\| & \leq\sum_{j=1}^{N}
  \left|D_{j}D^{\alpha}f(\xi)\right|
   \leq\sum_{j=1}^{N}\underset{|\beta|=1}{\sup}
  \left|D^{\beta}D^{\alpha}f(\xi)\right|
  =N\underset{|\beta|=1}{\sup}\left|D^{\alpha+\beta}f(\xi)\right|,
\end{align*}
so
\begin{align*}
  \sup_{\alpha\in\N_0^N}\sup_{x\in\tilde K}
  |(D^{\alpha}f)(tx)-D^{\alpha}f(x)|e^{-m\varphi^{*}\left(\frac{|\alpha|}{m}\right)}
  \leq & CN(1-t)\underset{\tilde{\alpha}\in\N^{N}}{\sup}
  \underset{\xi\in\ch(\tilde{K})}{\sup}\left|D^{\tilde{\alpha}}f(\xi)\right|
  e^{-m\varphi^{*}\left(\frac{|\alpha|}{m}\right)},
\end{align*}
where
$\underset{\tilde{\alpha}\in\N^{N}}{\sup}
\underset{\xi\in\ch(\tilde{K})}{\sup}
\left|D^{\tilde{\alpha}}f(\xi)\right|e^{-m\varphi^{*}\left(\frac{|\alpha|}{m}\right)}<\infty$
by definition of $f\in\E_{(\omega)}(\R^N)$. Then, from (\ref{eq:(8)}), we have
obtained (\ref{eq:(5)}), i.e. 
\[
f_{t}\rightarrow f\,\,\,\,\,\,\,\,\,\,\mbox{in}\ \E_{(\omega)}(\R^N).
\]
Therefore (\ref{eq:(4)}) holds true.

Since $\left\langle \mu,f_{t}\right\rangle =0$ for all $t\in(0,1)$,
then $\left\langle \mu,f\right\rangle =0$. This can be done for all
$f\in\E_{(\omega)}(\R^N)$ with $\left.f\right|_{K}=0$, and hence there exists 
$C,\lambda>0$
such that
\begin{equation}
  \left|\left\langle \mu,f\right\rangle \right|\leq Cp_{K,\lambda}(f)
  \,\,\,\,\,\,\,\,\,\,\forall f\in\E_{(\omega)}(\R^N).
  \label{eq:(9)}
\end{equation}
For $f_{z}(x)=e^{-i<x,z>}$, we observe that
\begin{align*}
  p_{K,\lambda} & (f_{z})=\sup_{\alpha\in\N_0^N}\sup_{x\in\R^N}
  |f_{z}^{(\alpha)}(x)|e^{-\lambda\varphi^*\left(\frac{|\alpha|}{\lambda}\right)}\leq
  e^{H_{K}(\Im z)+\lambda\omega(z)}
\end{align*}
by (\ref{eq:AA}).

Substituting in (\ref{eq:(9)}) with $f=f_{z}$ and remembering that
$\hat{\mu}(z)=\left\langle \mu,f_{z}\right\rangle $, we obtain (\ref{eq:(2)-1});
moreover $\hat{\mu}$ is entire because $f_{z}$ is entire
(cf. also \cite[Prop.~7.2]{BMT}).

To prove (\ref{eq:(3)-1}) we observe that if $\varphi\in\D_{(\omega)}(\R^N)$, then
\begin{align*}
  \left\langle \mu,\varphi\right\rangle  & =\mu*\check{\varphi}(0)
  ={\F}^{-1}\left(\widehat{\mu*\check{\varphi}}\right)(0)\\
 & ={\F}^{-1}\left(\hat{\mu}\cdot\hat{\check{\varphi}}\right)(0)\\
 & =\frac{1}{(2\pi)^{N}}\int_{\R^N}\hat{\mu}(t)\hat{\varphi}(-t)dt\\
 & =\frac{1}{(2\pi)^{N}}\int_{\R^N}\hat{\mu}(-t)\hat{\varphi}(t)dt.
\end{align*}
Conversely, let $g$ be entire on $\mathbb{C}^{N}$ statisfying (\ref{eq:(2)-1})
and define $\mu$ by
\[
\left\langle \mu,f\right\rangle =\frac{1}{(2\pi)^{N}}\int_{\R^N}g(-t)
\hat{f}(t)dt,\,\,\,\,\,\,\,\,\,\,f\in\D_{(\omega)}(\R^N).
\]
Then $\mu\in\D_{(\omega)}'(\R^N)$ with $\supp\mu\subset K$, hence
$\mu\in\E_{(\omega)}'(\R^N)$
by Proposition \ref{prop:5.3}, and $\hat{\mu}=g$ by (\ref{eq:(3)-1})
(see also \cite[Prop. 7.3]{BMT}).
\end{proof}

\begin{Prop}
  Let $\Omega\subset\R^N$ be open and $\omega\in\W$. For every $\mu\in
  \D_{\{\omega\}}'(\Omega)$
there is a weight function $\sigma\in\W$ with $\sigma=o(\omega)$
such that $\mu\in{\D}_{(\sigma)}'(\Omega)\subset\D_{\{\omega\}}'(\Omega)$.
The analogous statement holds for $\E_{\{\omega\}}'(\Omega)$.
\end{Prop}

\begin{proof}
See \cite{BMT}, Proposition 7.6.
\end{proof}

\section{The Cauchy problem for overdetermined systems.}

In this section we consider the Cauchy problem for overdetermined
systems of linear partial differential operators with constant coefficients
in the classes of $\omega-$ultradifferentiable functions of Beurling
type defined
in the previous section.

To bypass the question of formal coherence of the initial data, that
could be especially intricate in the overdetermined case (cf. \cite{AHLM},
\cite{AN2}, \cite{N2}), we consider initial data in the Whitney
sense, in the spirit of \cite{N2}, \cite{BN1}, \cite{BN3}.

Let $K_{1}$ and $K_{2}$ be closed and convex subsets of $\R^N$ such
that $K_{1}\subsetneq K_{2}$ and $\overline{\mathring{K_{j}}}=K_{j}$
for $j=1,2$.

For $\omega\in\W'$ we denote by ${\I}^{(\omega)}(K_{2},\Omega)$ the subspace of
functions in $\E_{(\omega)}(\Omega)$ which vanish of infinite order on $K_{2}:$
\[
  {\I}^{(\omega)}(K_{2},\Omega)=\left\{ f\in\E_{(\omega)}(\Omega):\,
  D^{\alpha}f\equiv0\,\,\mbox{on}\,\,K_{2},\,\forall\alpha\in\N_{0}^{N}\right\} .
\]

\begin{Def}
  \begin{em}
    Let $\omega\in\W'$.
We define the space
$W_{K_{2}}^{(\omega)}$ of \emph{Whitney $\omega$-ultradifferentiable functions}
on $K_{2}$ by the exact sequence
\[
0\longrightarrow{\I}^{(\omega)}(K_{2},\Omega)\longrightarrow
\E_{(\omega)}(\Omega)\longrightarrow W_{K_{2}}^{(\omega)}\longrightarrow0;
\]
i.e.
\[
W_{K_{2}}^{(\omega)}\simeq\E_{(\omega)}(\Omega)\bigl/{\I}^{(\omega)}(K_{2},\Omega).
\]
In the same way we define $W_{K_{1}}^{(\omega)}$, the space of Whitney
$\omega$-ultradifferentiable functions on $K_{1}$.
\end{em}
\end{Def}

Denoting by $\mathcal P=\C[\zeta_1,\ldots,\zeta_N]$ the ring
of complex polynomials
in $N$ indeterminates, we consider $W_{K_2}^{(\omega)}$ as a 
unitary left and right ${\mathcal P}-$module
by the action of $p(\zeta)$ on $u\in W_{K_2}^{(\omega)}$ described by
\[
p(\zeta)u=up(\zeta)=p(D)u,
\]
by the formal substitution $\zeta_j\leftrightarrow\frac1i\partial_j$.

Given an $a_1\times a_0$ matrix $A_0(D)$ with polynomial entries, we can
thus consider the corresponding operator $A_0(D)$.
We want to solve, in the Whitney's sense, the Cauchy problem
\beqs
\label{G3}
\begin{cases}
A_{0}(D)u=f\\
\left.u\right|_{K_{1}}\equiv0,
\end{cases}
\eeqs
where $\left.u\right|_{K_{1}}\equiv0$ means that $u$ vanishes with
all its derivatives on
$K_1$.

Let us remark that if
$^{t}Q(\zeta):{\mathcal P}^{a_{1}}\rightarrow{\mathcal P}$ is such
that 
\begin{equation}
  ^{t}A_{0}(\zeta)^{t}Q(\zeta)\equiv0,
  \label{eq:1viola}
\end{equation}
then, in order to solve the Cauchy problem \eqref{G3},
$f$ must satisfy the integrability condition
\begin{equation}
  Q(D)f=0,
  \label{eq:2viola}
\end{equation}
because of $Q(D)f=Q(D)A_{0}(D)u=0$.

Since ${\mathcal P}$ is a Noetherian ring, the collection of all vectors
$^{t}Q(\zeta)$ satisfying (\ref{eq:1viola}) form a finitely generated
${\mathcal P}-$module. So we can insert the map
$^{t}A_{0}(\zeta):\ {\mathcal P}^{a_{1}}\rightarrow{\mathcal P}^{a_{0}}$
into a Hilbert resolution:
\[
0\longrightarrow{\mathcal P}^{a_{d}}\xrightarrow{^{t}A_{d-1}(\zeta)}
{\mathcal P}^{a_{d-1}}\longrightarrow\ldots\longrightarrow{\mathcal P}^{a_{2}}
\xrightarrow{^{t}A_{1}(\zeta)}{\mathcal P}^{a_{1}}\xrightarrow{^{t}A_{0}(\zeta)}
            {\mathcal P}^{a_{0}}\longrightarrow{\M}\longrightarrow0,
\]
where
${\M}=\coker{}^{t}A_{0}(\zeta)={\mathcal P}^{a_{0}}\bigl/{}^{t}A_{0}(\zeta)
{\mathcal P}^{a_{1}}$
and the matrix $^{t}A_{1}(\zeta)$ is obtained from a basis of the
integrability conditions (\ref{eq:1viola}). The sequence is exact,
i.e. $\Im{}^{t}A_{j}=\Ker{}^{t}A_{j+1}$.

Therefore a necessary condition to solve \eqref{G3}, is that
$f$ satisfies the following integrability condition:
\begin{equation}
  A_{1}(D)f=0.
  \label{eq:4viola}
\end{equation}

Moreover, every necessary condition for the solvability of \eqref{G3},
which can be expressed in terms of linear partial differential equations,
is a consequence of (\ref{eq:4viola}) (cf. \cite{N2}).

\begin{Def}
  \begin{em}
If $A_{1}(\zeta)\not\not\equiv0$ then
the Cauchy problem \eqref{G3} is called \emph{overdetermined} and,
to solve it, the condition (\ref{eq:4viola}) has to be satisfied.
\end{em}
  \end{Def}

Let us remark that if $u$ solves \eqref{G3}, then also $f$ must vanish
with all its
derivatives on $K_1$, so that we look for solutions
$u\in \left(W_{K_2}^{(\omega)}\right)^{a_0}$ of \eqref{G3} when $f$
satisfies
\beqs
\label{GG3}
\begin{cases}
f\in\left(W_{K_{2}}^{(\omega)}\right)^{a_{1}}\\
A_{1}(D)f=0\\
\left.f\right|_{K_{1}}\equiv0.
\end{cases}
\eeqs

\begin{Rem}
  \begin{em}
By Whitney's extension theorem it's not restrictive to consider zero
Cauchy data (see \cite{B}, \cite{BBMT}, \cite{MT}, \cite{M}).
\end{em}
\end{Rem}

Let us denote by ${\I}^{(\omega)}(K_{1,}K_{2})$ the space of
Whitney $\omega$-ultradifferentiable fuctions on $K_{2}$ which vanish
of infinite
order on $K_{1}:$
\[
  {\I}^{(\omega)}(K_{1,}K_{2})=\left\{ f\in W_{K_{2}}^{(\omega)}:\,
  D^{\alpha}\left.f\right|_{K_{1}}\equiv0\,\,\forall\alpha\in\N_{0}^{N}\right\} .
\]

The Cauchy problem \eqref{G3}-\eqref{GG3} is then equivalent to:
\begin{equation}
\begin{cases}
  \mbox{given}\,\,f\in\left({\I}^{(\omega)}(K_{1},K_{2})\right)^{a_{1}}\,\,
  \mbox{such that}\,\,A_{1}(D)f=0\\
  \mbox{find}\,\,u\in\left({\I}^{(\omega)}(K_{1},K_{2})\right)^{a_{0}}\,\,
  \mbox{such that}\,\,A_{0}(D)u=f.
\end{cases}
\label{eq:sistema whitney}
\end{equation}

\begin{Rem}
\label{rem:villa}
  \begin{em}
By the isomorphisms
\beqsn
&&\Ext_{\mathcal P}^{0}\left({\M},{\I}^{(\omega)}(K_{1},K_{2})\right)\simeq
\Ker A_0(D)=\{u\in{\I}^{(\omega)}(K_{1},K_{2})^{a_0}:\ A_0(D)u=0\}\\
&&\Ext_{\mathcal P}^{1}\left({\M},{\I}^{(\omega)}(K_{1},K_{2})\right)\simeq
\frac{\Ker A_{1}(D)}{\Im A_{0}(D)},
\eeqsn
we have:
\begin{enumerate}
  \item
    \textbf{uniqueness} of solutions of the Cauchy problem
    (\ref{eq:sistema whitney})
is equivalent to the condition 
\[
\Ext_{{\mathcal P}}^{0}({\M},{\I}^{(\omega)}(K_{1},K_{2}))=0;
\]
\item
\textbf{existence }of solutions of (\ref{eq:sistema whitney}), is
equivalent to the condition
\[
\Ext_{{\mathcal P}}^{1}({\M},{\I}^{(\omega)}(K_{1},K_{2}))=0;
\]
\item
\textbf{existence and uniqueness} of a solution of (\ref{eq:sistema whitney}),
is equivalent to the condition
\[
\Ext_{{\mathcal P}}^{0}({\M},{\I}^{(\omega)}(K_{1},K_{2}))=
\Ext_{{\mathcal P}}^{1}({\M},{\I}^{(\omega)}(K_{1},K_{2}))=0.
\]
\end{enumerate}
\end{em}
\end{Rem}

\begin{Rem}
   \begin{em}
     Remark~\ref{rem:villa} enlightens the algebraic invariance of the
     problem: uniqueness and/or existence of solutions of the Cauchy problem
     \eqref{eq:sistema whitney} depend only on the module $\mathcal M$
     and not on it's presentation by a particular matrix ${}^tA_0(D)$.
  \end{em}
\end{Rem}

Note also that we have the short exact sequence
\[
0\longrightarrow{\I}^{(\omega)}(K_{1},K_{2})\longrightarrow W_{K_{2}}^{(\omega)}
\longrightarrow W_{K_{1}}^{(\omega)}\longrightarrow0,
\]
that implies the long exact sequence
\begin{align}
  0\longrightarrow\Ext_{{\mathcal P}}^{0}\left({\M},{\I}^{(\omega)}(K_{1},K_{2})\right)
  & \longrightarrow\Ext_{{\mathcal P}}^{0}\left({\M},W_{K_{2}}^{(\omega)}\right)
  \longrightarrow\Ext_{{\mathcal P}}^{0}\left({\M},W_{K_{1}}^{(\omega)}\right)
  \longrightarrow
  \nonumber \\
  \longrightarrow\Ext_{{\mathcal P}}^{1}\left({\M},{\I}^{(\omega)}(K_{1},K_{2})\right)
  & \longrightarrow\Ext_{{\mathcal P}}^{1}\left({\M},W_{K_{2}}^{(\omega)}\right)
  \longrightarrow\Ext_{{\mathcal P}}^{1}\left({\M},W_{K_{1}}^{(\omega)}\right)
  \longrightarrow
  \nonumber \\
  \longrightarrow\Ext_{{\mathcal P}}^{2}\left({\M},{\I}^{(\omega)}(K_{1},K_{2})\right)
  & \longrightarrow\ldots.
  \label{eq:3312.}
\end{align}

As in \cite{N1} (cf. also \cite{B,BN1,BN3}) we have
that $W_{K_{i}}^{(\omega)}$, for $i=1,\,2$, are injective ${\mathcal P}-$modules,
i.e. the following holds:

\begin{Lemma}
\label{thm:A1.5-1}
  Let $\omega\in\W'$, ${\M}$ a ${\mathcal P}-$module
of finite type and $K$ a compact convex subset of $\R^N$.
Then $\Ext_{{\mathcal P}}^{j}({\M},W_{K_{i}}^{(\omega)})=0$
for $i=1,\,2$ and for all $j\geq1$.
\end{Lemma}

By Lemma~\ref{thm:A1.5-1}, the complex (\ref{eq:3312.}) reduces
to:
\begin{align*}
  0 & \longrightarrow\Ext_{{\mathcal P}}^{0}\left({\M},{\I}^{(\omega)}(K_{1},K_{2})
  \right)\longrightarrow\Ext_{{\mathcal P}}^{0}\left({\M},W_{K_{2}}^{(\omega)}\right)
  \longrightarrow\Ext_{{\mathcal P}}^{0}\left({\M},W_{K_{1}}^{(\omega)}\right)
  \longrightarrow\\
  & \longrightarrow\Ext_{{\mathcal P}}^{1}\left({\M},{\I}^{(\omega)}(K_{1},K_{2})\right)
  \longrightarrow0.
\end{align*}
In particular 
\begin{equation}
  \Ext_{{\mathcal P}}^{j}({\M},{\I}^{(\omega)}(K_{1},K_{2}))=0
  \,\,\,\,\,\,\,\,\,\,\forall j>1.
  \label{eq:(7)}
\end{equation}

\begin{Rem}
\label{remG11}
\begin{em}
From Remark~\ref{rem:villa} and the above considerations, it follows that 
uniqueness and/or existence of solutions of the Cauchy problem 
\eqref{eq:sistema whitney} is related to injectivity and/or surjectivity
of the homomorphism
\beqs
\label{G12}
\Ext_{\mathcal P}^{0}\left(\M,W_{K_{2}}^{(\omega)}\right)
\longrightarrow\Ext_{\mathcal P}^{0}\left(\M,W_{K_{1}}^{(\omega)}\right).
\eeqs

The injectivity of \eqref{G12} is equivalent to the fact that
the dual homomorphism
\beqs
\label{G13}
\Ext_{\mathcal P}^{0}\left(\M,W_{K_{1}}^{(\omega)}\right)'
\longrightarrow\Ext_{\mathcal P}^{0}\left(\M,W_{K_{2}}^{(\omega)}\right)'
\eeqs
has a dense image.

Moreover, surjectivity is equivalent to have a dense and closed image.
But \eqref{G12} has a closed image if and only if \eqref{G13} has
a closed image (cf. \cite{Gr}, Ch.~IV, \S~2, n.~4, Thm.~3), so that
the surjectivity of \eqref{G12} is equivalent to the fact that the dual 
homorphism \eqref{G13} is injective and has a closed image.
\end{em}
\end{Rem}

By Remarks~\ref{rem:villa} and \ref{remG11},
and \cite[Prop. 1.1-1.2]{N2}, we have that:

\begin{Prop}
  Let $\omega\in\W'$ and $K_{1},\,K_{2}$ closed convex subsets of $\R^N$ with
  $K_{1}\subsetneq K_{2}$,
  $\overline{\mathring{K}_{j}}=K_{j}$
  for $j=1,\,2$.
  Let
$\M$ be a unitary ${\mathcal P}-$module of finite
  type and denote by $\Ass(\M)$ the set of all prime ideals
  associated to $\M$.

  Then the following statements are equivalent:
\begin{enumerate}
  \item
The Cauchy problem (\ref{eq:sistema whitney}) admits \textbf{at most one}
solution;
\item
$\ds\Ext_{{\mathcal P}}^{0}({\M},{\I}^{(\omega)}(K_{1},K_{2}))=0;$
\item
$\ds\Ext_{{\mathcal P}}^{0}({\mathcal P}\bigl/\wp,{\I}^{(\omega)}(K_{1},K_{2}))=0$
~for all $\wp\in\Ass(\M)$;
\item
The homomorphisms
\[
\Ext_{{\mathcal P}}^{0}\left({\mathcal P}\bigl/\wp,W_{K_{2}}^{(\omega)}\right)
\rightarrow\Ext_{{\mathcal P}}^{0}\left({\mathcal P}\bigl/\wp,W_{K_{1}}^{(\omega)}\right)
\]
are injective for all $\wp\in\Ass(\M)$;
\item
The homomorphisms
\[
\Ext_{{\mathcal P}}^{0}\left({\mathcal P}\bigl/\wp,W_{K_{1}}^{(\omega)}\right)'
\rightarrow\Ext_{{\mathcal P}}^{0}\left({\mathcal P}\bigl/\wp,W_{K_{2}}^{(\omega)}
\right)'
\]
have a dense image for all $\wp\in\Ass(\M)$.
\end{enumerate}
\end{Prop}

\begin{Prop}
  \label{prop:esistenza}
  Let $\omega\in\W'$, $\M$ a unitary ${\mathcal P}-$module of
finite type and $K_{1},\,K_{2}$ closed convex subsets of $\R^N$
with $K_{1}\subsetneq K_{2}$, $\overline{\mathring{K}_{j}}=K_{j}$
for $j=1,\,2$. Then the following statements are equivalent:
\begin{enumerate}
\item
  The Cauchy problem (\ref{eq:sistema whitney}) admits
\textbf{at least }a solution;
\item
$\ds\Ext_{{\mathcal P}}^{1}({\M},{\I}^{(\omega)}(K_{1},K_{2}))=0;$
\item
$\ds\Ext_{{\mathcal P}}^{1}({\mathcal P}\bigl/\wp,{\I}^{(\omega)}(K_{1},K_{2}))=0$
for all $\wp\in\Ass(\M)$;
\item
The homomorphisms
\beqs
\label{G4}
\Ext_{{\mathcal P}}^{0}\left({\mathcal P}\bigl/\wp,W_{K_{2}}^{(\omega)}\right)
\rightarrow\Ext_{{\mathcal P}}^{0}\left({\mathcal P}\bigl/\wp,
W_{K_{1}}^{(\omega)}\right)
\eeqs
are surjective for all $\wp\in\Ass(\M)$;
\item
The homomorphisms
\beqs
\label{eq:9}
\Ext_{{\mathcal P}}^{0}\left({\mathcal P}\bigl/\wp,W_{K_{1}}^{(\omega)}\right)'
\rightarrow\Ext_{{\mathcal P}}^{0}\left({\mathcal P}\bigl/\wp,W_{K_{2}}^{(\omega)}
\right)'
\eeqs
are injective and have a closed image, for all $\wp\in\Ass(\M)$.
\end{enumerate}
\end{Prop}

\begin{Prop}
  \label{prop:3.3.7}
  Let $\omega\in\W'$, $\M$ a unitary ${\mathcal P}-$module of finite
type and $K_{1},\,K_{2}$ closed convex subsets of $\R^N$ with
$K_{1}\subsetneq K_{2}$, $\overline{\mathring{K}_{j}}=K_{j}$ for
$j=1,\,2$. Then the following statements are equivalent:
\begin{enumerate}
  \item
The Cauchy problem (\ref{eq:sistema whitney}) admists
\textbf{one and only one} solution;
\item
$\ds\Ext_{{\mathcal P}}^{0}({\M},{\I}^{(\omega)}(K_{1},K_{2}))=
\Ext_{{\mathcal P}}^{1}({\M},{\I}^{(\omega)}(K_{1},K_{2}))=0;$
\item
$\ds\Ext_{{\mathcal P}}^{0}({\mathcal P}\bigl/\wp,{\I}^{(\omega)}(K_{1},K_{2}))
=\Ext_{{\mathcal P}}^{1}({\mathcal P}\bigl/\wp,{\I}^{(\omega)}(K_{1},K_{2}))=0$
for all $\wp\in\Ass(\M);$
\item
The homomorphisms
\[
\Ext_{{\mathcal P}}^{0}\left({\mathcal P}\bigl/\wp,W_{K_{2}}^{(\omega)}\right)
\rightarrow\Ext_{{\mathcal P}}^{0}\left({\mathcal P}\bigl/\wp,
W_{K_{1}}^{(\omega)}\right)
\]
are isomorphisms for all $\wp\in\Ass(\M)$;
\item
The homomorphisms
\[
\Ext_{{\mathcal P}}^{0}\left({\mathcal P}\bigl/\wp,W_{K_{1}}^{(\omega)}\right)'
\rightarrow\Ext_{{\mathcal P}}^{0}\left({\mathcal P}\bigl/\wp,W_{K_{2}}^{(\omega)}
\right)'
\]
are isomorphisms for all $\wp\in\Ass(\M)$.
\end{enumerate}
\end{Prop}

The overdetermined Cauchy
problem (\ref{eq:sistema whitney}) is thus reduced to the study of the dual
homomorphism 
\beqsn
\Ext_{{\mathcal P}}^{0}\left({\mathcal P}\bigl/\wp,W_{K_{1}}^{(\omega)}\right)'
\rightarrow\Ext_{{\mathcal P}}^{0}\left({\mathcal P}\bigl/\wp,W_{K_{2}}^{(\omega)}
\right)',\,\,\,\,\,\,\,\,\,\,\,\wp\in\Ass(\M).
\eeqsn

Let us start with some preliminary results.

\begin{Lemma}
\label{lem:Lemma1}
Let $\omega\in\W'$ and $K$ a convex and closed subset of $\R^N$,
with $\overline{\mathring{K}}=K$. Then
\[
\left(W_{K}^{(\omega)}\right)'\simeq\E_{(\omega)}'(K).
\]
\end{Lemma}

\begin{proof}
If $K$ is compact with $\mathring{K}\not=\emptyset$ and $0\in\mathring{K}$,
then by \cite[Cor. 4.7]{BBMT} and \cite[Prop. 3.6]{MT}
we have that $\mu\in\left(W_{K}^{(\omega)}\right)'$ if and only if
there exist $\lambda,\,C>0$ such that
\begin{equation}
|\hat{\mu}(\zeta)|\leq Ce^{\left(H_{K}(\Im\zeta)+\lambda\omega(\zeta)\right)}
\,\,\,\,\,\,\,\,\,\forall\zeta\in\mathbb{C}^{N}.
\label{eq:stella-1}
\end{equation}
By the Paley-Wiener Theorem \ref{thm:prop 7.2-7.3}, this is equivalent
to
\[
\mu\in\E_{(\omega)}'(K).
\]
Therefore the lemma is proved 
(cf. also \cite{M}, \cite{N1}).
\end{proof}

\begin{Lemma}
\label{lem:lemma2}
Let $\omega\in\W'$, $\wp$ a prime ideal of $\mathcal P$ and
$K\subset\R^N$ a convex and closed set with 
$\overline{\mathring{K}}=K$. Then
we have the following isomorphism: 
\begin{equation}
\Ext_{{\mathcal P}}^{0}\left({\mathcal P}\bigl/\wp,W_{K}^{(\omega)}\right)'
\simeq\E_{(\omega)}'(K)\bigl/\wp(D)\otimes\E_{(\omega)}'(K)
\label{eq:ISOMORFISMO}
\end{equation}
with 
\[
\wp(D)\otimes\E_{(\omega)}'(K):=\left\{\sum_{h=1}^{r}p_{h}(D)T_{h}:\,\,
T_{h}\in\E_{(\omega)}'(K)\right\} ,
\]
where $p_{1}(\zeta),\ldots,p_{r}(\zeta)$ are generators of $\wp$.
\end{Lemma}

\begin{proof}
For any closed subspace $F$ of a Fr\'echet space $E$, the dual $F'$
of $F$ is isomorphic (cf. \cite[Prop.~6.14]{MV}) to:
\[
F'\simeq E'\bigl/F^{\text{0}},
\]
where $F^{0}$ is the annihilator of $F$, defined by
\[
F^{0}:=\left\{ T\in E':\,T(f)=0,\,\,\,\forall f\in F\right\} .
\]
Then, since 
$\Ext_{{\mathcal P}}^{0}\left({\mathcal P}\bigl/\wp,W_{K}^{(\omega)}\right)$
is a closed subspace of the Fr\'echet space $W_{K}^{(\omega)}$, we
have
\begin{align*}
\Ext_{{\mathcal P}}^{0}\left({\mathcal P}\bigl/\wp,W_{K}^{(\omega)}\right)' 
& \simeq\left(W_{K}^{(\omega)}\right)'\bigl/\left(\Ext_{{\mathcal P}}^{0}
\left({\mathcal P}\bigl/\wp,W_{K}^{(\omega)}\right)\right)^{0},
\end{align*}
and, by Lemma \ref{lem:Lemma1},
\beqs
\label{G6}
\Ext_{{\mathcal P}}^{0}\left({\mathcal P}\bigl/\wp,W_{K}^{(\omega)}\right)' 
& \simeq\E_{(\omega)}'(K)\bigl/\left(\Ext_{{\mathcal P}}^{0}\left({\mathcal P}
\bigl/\wp,W_{K}^{(\omega)}\right)\right)^{0},
\eeqs
with
\[
\left(\Ext_{{\mathcal P}}^{0}\left({\mathcal P}\bigl/\wp,W_{K}^{(\omega)}\right)
\right)^{0}=\left\{ T\in\E_{(\omega)}'(K):\,\,T(u)=0,\,\,\,
\forall u\in\Ext_{{\mathcal P}}^{0}\left({\mathcal P}\bigl/\wp,W_{K}^{(\omega)}
\right)\right\} .
\]

Observe that, for $V(\wp)$ defined as in (\ref{Vwp}), we have
\beqsn
V=V(\wp)=\{\zeta\in\C^N:\ p_h(-\zeta)=0,\ \forall h=1,\ldots,r\}
\eeqsn
and 
\[
p_{h}(D_{x})e^{-i<x,\zeta>}=p_{h}(-\zeta)e^{-i<x,\zeta>}=0
\,\,\,\,\,\,\,\,\,\,\forall\zeta\in V(\wp).
\]
But
\begin{align*}
\Ext_{{\mathcal P}}^{0}\left({\mathcal P}\bigl/\wp,W_{K}^{(\omega)}\right) 
& =\Ker A_{0}(D)\\
 & =\left\{ u\in W_{K}^{(\omega)}:\,p_{h}(D)u=0\,\,\,\forall h=1,\ldots,r\right\},
\end{align*}
so that
\begin{equation}
e^{-i<\cdot,\zeta>}\in\Ext_{{\mathcal P}}^{0}\left({\mathcal P}\bigl/\wp,
W_{K}^{(\omega)}\right)\Leftrightarrow\zeta\in V(\wp).
\label{eq:16}
\end{equation}

Therefore the Fourier-Laplace transform $\hat{T}(\zeta)$ of an element
$T\in\left(\Ext_{{\mathcal P}}^{0}\left({\mathcal P}\bigl/\wp,W_{K}^{(\omega)}
\right)\right)^{0}$
is an entire function which satisfies:
\[
\hat{T}(\zeta)=\left\langle T,e^{-i<\cdot,\zeta>}\right\rangle =0
\,\,\,\,\,\,\,\,\,\,\,\forall\zeta\in V(\wp).
\]

By the Nullstellensatz (see \cite{T}), there exist entire functions
$F_{1}(\zeta),\ldots,F_{r}(\zeta)$ such that 
\[
\hat{T}(\zeta)=\sum_{h=1}^{r}p_{h}(-\zeta)F_{h}(\zeta)
\,\,\,\,\,\,\,\,\,\forall\zeta\in\mathbb{C}^{N}.
\]

By the Paley-Wiener Theorem \ref{thm:prop 7.2-7.3}, the Fourier-Laplace
transform of a distribution $T\in\E_{(\omega)}'(K)$ is characterized by an
estimate of the form
\begin{equation}
\left|\hat{T}(\zeta)\right|\leq Ce^{H_{\sigma_{T}}(\Im\zeta)+\alpha\omega(\zeta)},
\label{eq:(10)}
\end{equation}
for some $C>0$, $\alpha\in\N$,
where $\sigma_{T}\subset K$ is the convex hull of $\supp T$.

If $K$ is not compact, we choose $K_{\alpha}\subset\mathring{K}_{\alpha+1}$
compact and such that $K=\underset{\alpha}{\cup}K_{\alpha}$, while
if $K$ is compact, we choose $K_{\alpha}=K$ for all $\alpha$.

Since $\sigma_{T}\subset K_{\alpha}$ for some $\alpha$, then (\ref{eq:(10)})
implies that there exist $C>0$, $\alpha\in\N$ such that
\begin{equation}
\left|\hat{T}(\zeta)\right|\leq Ce^{H_{K_{\alpha}}(\Im\zeta)+\alpha\omega(\zeta)}.
\label{eq:10'}
\end{equation}
Define
\[
\psi_{\alpha}(\zeta):=H_{K_{\alpha}}(\Im\zeta)+\alpha\omega(\zeta);
\]
since $\omega$ is plurisubharmonic by condition $(\delta)$, then
$\psi_{\alpha}(\zeta)$ is plurisubharmonic in $\mathbb{C}^{N}$.

Moreover we have that for every $k_{0}>0$ there exists $k_{1}>0$
such that 
\begin{equation}
\left|\psi_{\alpha}(\Im z+\Im\zeta)-\psi_{\alpha}(\Im\zeta)\right|
\leq k_{1}\,\,\,\,\,\,\,\,\,\,\,\,\,\,\,\,{\rm for}\,\,
|z|\leq k_{0}.
\label{eq:marco}
\end{equation}

Indeed,
\begin{align*}
\left|\psi_{\alpha}(z+\zeta)-\psi_{\alpha}(\zeta)\right| & =
\left|H_{K}(\Im z+\Im\zeta)+\alpha\omega(z+\zeta)-H_{K}(\Im\zeta)-
\alpha\omega(\zeta)\right|\\
 & \leq\left|H_{K}(\Im z+\Im\zeta)-H_{K}(\Im\zeta)\right|+\alpha
\left|\omega(z+\zeta)-\omega(\zeta)\right|.
\end{align*}
Now observe that 
\begin{align*}
H_{K}(\Im z+\Im\zeta)-H_{K}(\Im\zeta) & \leq H_{K}(\Im z)+H_{K}(\Im\zeta)
-H_{K}(\Im\zeta)\leq c\,\,\,\,\,\,\,\,\,|z|\leq k_{0},
\end{align*}
for some $c>0$ and
\begin{align}
H_{K}(\Im\zeta)-H_{K}(\Im z+\Im\zeta) & \leq H_{K}(\Im\zeta)
-\left\langle x,\Im z\right\rangle -\left\langle x,\Im\zeta\right\rangle 
\,\,\,\,\,\,\,\,\,\forall x\in K.
\label{eq:1.*-1}
\end{align}

Moreover, by definition of supremum, for all $\varepsilon>0$ there exists
$\bar{x}\in K$ such that
\[
\left\langle \bar{x},\Im\zeta\right\rangle >H_{K}(\Im\zeta)-\varepsilon.
\]
So, choosing such $\bar{x}$ in (\ref{eq:1.*-1}) we have
\[
H_{K}(\Im\zeta)-H_{K}(\Im z+\Im\zeta)\leq\varepsilon+c',
\qquad\mbox{if}\ |z|\leq k_{0},
\]
for some $c'>0$, hence there exists $k_{1}>0$ such that
\[
\left|H_{K}(\Im z+\Im\zeta)-H_{K}(\Im\zeta)\right|\leq k_{1},
\qquad|z|\leq k_{0}.
\]

Furthermore, by Lemma~\ref{lemma12BMT} we have that
\[
\omega(z+\zeta)\leq K(1+\omega(z)+\omega(\zeta)),
\]
for some $K>0$ and hence for every $k_{0}>0$ there exists $k'_{1}>0$
such that 
\[
\left|\omega(z+\zeta)-\omega(\zeta)\right|\leq k'_{1},
\,\,\,\,\,\,\,\,\,\,\,|z|\leq k_{0}.
\]
Therefore (\ref{eq:marco}) is proved. 

We can therefore apply the Ehrenpreis Fundamental Theorem (see
\cite[Thm. 7.7.13]{H3}, and \cite{B,G} for more details)
and obtain that we can choose the entire functions $F_{h}$ satisfying
\begin{align*}
\left|F_{h}(\zeta)\right| & \leq C'e^{H_{K_{\alpha}}(\Im\zeta)+\alpha\omega(\zeta)
+m'\log(1+|\zeta|)},
\end{align*}
for some $C'>0$, $m'\in\N$.

By condition $\left(\gamma\right)'$
\[
m'\log(1+|\zeta|)\leq\frac{m'}{b}\omega(\zeta)-\frac{m'a}{b},
\]
so there exist $C'',\,C'''>0$ and $\alpha'\in\N$ such that 
\begin{align*}
\left|F_{h}(\zeta)\right| & \leq C''e^{H_{K_{\alpha}}(\Im\zeta)+\alpha'\omega(\zeta)}
\leq C'''e^{H_{K_{\alpha''}}(\Im\zeta)+\alpha''\omega(\zeta)}
\end{align*}
with $\alpha''=\max\{\alpha,\alpha'\}$.

Hence, by the Paley-Wiener Theorem \ref{thm:prop 7.2-7.3}:
\[
F_{h}=\widehat{T_{h}}
\]
for some $T_{h}\in\E_{(\omega)}'(K)$.

We have thus proved that if $T\in\left(\Ext_{{\mathcal P}}^{0}\left({\mathcal P}
\bigl/\wp,W_{K}^{(\omega)}\right)\right)^{0}$,
then
\begin{align*}
\hat{T}(\zeta)  =\sum_{h=1}^{r}p_{h}(-\zeta)\widehat{T_{h}}(\zeta)
  =\sum_{h=1}^{r}\widehat{p_{h}(D)T_{h}(\zeta)},\qquad
\mbox{with }\,\,T_{h}\in\E_{(\omega)}'(K).
\end{align*}

This result implies that
\[
T\in\wp(D)\otimes\E_{(\omega)}'(K),
\]
and so, by \eqref{G6},
\[
\left(\Ext_{{\mathcal P}}^{0}\left({\mathcal P}\bigl/\wp,W_{K}^{(\omega)}\right)
\right)'\simeq\E_{(\omega)}'(K)\bigl/\wp(D)\otimes\E_{(\omega)}'(K).
\]
\end{proof}

Let us define ${\mathcal O}_{\psi_{\alpha}}(\mathbb{C}^{N})$ as the space
of holomorphic functions $u$ on $\mathbb{C}^{N}$ which satisfy for
some $C>0$ and for all $\zeta\in\mathbb{C}^{N}:$
\begin{equation}
\left|u(\zeta)\right|\leq Ce^{\psi_{\alpha}(\zeta)}=
Ce^{H_{K_{\alpha}}(\Im\zeta)+\alpha\omega(\zeta)}.
\label{eq:12}
\end{equation}

We can then consider the inductive limit 
\[
{\mathcal O}_{\psi}(\mathbb{C}^{N}):=\underset{\alpha
\rightarrow\infty}{\indlim}{\mathcal O}_{\psi_{\alpha}}(\mathbb{C}^{N}).
\]

From the Paley-Wiener Theorem \ref{thm:prop 7.2-7.3}, by Fourier-Laplace
transform we have the following isomorphism:
\[
\E_{(\omega)}'(K)\simeq{\mathcal O}_{\psi}(\mathbb{C}^{N}).
\]

Therefore, from Lemma \ref{lem:lemma2}: 
\beqs
\label{G7}
\Ext_{{\mathcal P}}^{0}\left({\mathcal P}\bigl/\wp,W_{K}^{(\omega)}\right)'
\simeq{\mathcal O}_{\psi}(\mathbb{C}^{N})\bigl/\wp(D)
\otimes{\mathcal O}_{\psi}(\mathbb{C}^{N}).
\eeqs

Let $V$ be a reduced affine algebraic variety. Denote by 
${\mathcal O}_{\psi_{\alpha}}(V)$
the space of holomorphic functions on $V$ (i.e. complex valued continuous
functions on $V$ which are restrictions of entire functions on $\mathbb{C}^{N})$
that satisfy (\ref{eq:12}) for some $\alpha\in\N$, $C>0$ and for
all $\zeta\in V$. Consider then the inductive limit 
\[
{\mathcal O}_{\psi}(V)=
\underset{\alpha\rightarrow\infty}{\indlim}{\mathcal O}_{\psi_{\alpha}}(V).
\]
We have the following:

\begin{Prop}
\label{lem:LEMMA3}
Let $\omega\in\W'$, $\wp$ a prime ideal of ${\mathcal P}$ with
associated algebraic variety $V=V(\wp)$, and $K$ a closed convex
subset of $\R^N$ with $\overline{\mathring{K}}=K$. Then we have a
natural isomorphism:
\[
\Ext_{{\mathcal P}}^{0}\left({\mathcal P}\bigl/\wp,W_{K}^{(\omega)}\right)'
\simeq{\mathcal O}_{\psi}(V).
\]
\end{Prop}

\begin{proof}
  By \eqref{G7} we have to prove the following isomorphism:
\[
{\mathcal O}_{\psi}(\mathbb{C}^{N})\bigl/\wp(D)\otimes{\mathcal O}_{\psi}
(\mathbb{C}^{N})
\simeq{\mathcal O}_{\psi}(V).
\]

First of all we prove that the homomorphism 
\begin{equation}
{\mathcal O}_{\psi}(\mathbb{C}^{N})\bigl/\wp(D)\otimes{\mathcal O}_{\psi}
(\mathbb{C}^{N})
\rightarrow{\mathcal O}_{\psi}(V)
\label{eq:omomo}
\end{equation}
is injective: if $f\in{\mathcal O}_{\psi}(\mathbb{C}^{N})$ is zero on
$V$, then by the Nullstellensatz there exist entire functions $f_{h}$ on
$\mathbb{C}^{N}$ such that
\[
f(\zeta)=\sum_{h=1}^{r}p_{h}(-\zeta)f_{h}(\zeta)
\,\,\,\,\,\,\,\,\,\forall\zeta\in\mathbb{C}^{N}.
\]

Since $f$ satisfies (\ref{eq:12}) by assumption, from the Ehrenpreis 
Fundamental Theorem \cite[Thm. 7.7.13]{H3} (see also \cite{B}, \cite{G}
for more details), we can choose $f_{h}$ satisfying (\ref{eq:12})
too, hence $f_{h}\in{\mathcal O}_{\psi}(\mathbb{C}^{N})$ and this implies
that $f\in\wp\otimes{\mathcal O}_{\psi}(\mathbb{C}^{N})$. So we have
obtained that $f$ is the zero element of 
${\mathcal O}_{\psi}(\mathbb{C}^{N})\bigl/
\wp(D)\otimes{\mathcal O}_{\psi}(\mathbb{C}^{N})$,
proving the injectivity of the homomorphism (\ref{eq:omomo}).

On the other hand, the homomorphism (\ref{eq:omomo}) is surjective:
if $f\in{\mathcal O}_{\psi}(V)$, then $f\in{\mathcal O}(\mathbb{C}^{N})$
and satisfies (\ref{eq:12}) for some $\alpha\in\N$, $C>0$ and for
all $\zeta\in V$. By the Ehrenpreis Fundamental Theorem \cite[Thm. 7.7.13]{H3},
there exist $g\in{\mathcal O}(\mathbb{C}^{N})$, with $f=g$ on $V$,
and two constants $C'>0$ and $n\in\N$ such that 
\[
\underset{\mathbb{C}^{N}}{\sup}|g|e^{-\psi_{\alpha}-n\log(1+|\zeta|)}
\leq C'\underset{V}{\sup}|f|e^{-\psi_{\alpha}}.
\]

Since the right-hand side is finite because f satisfies
(\ref{eq:12}) on $V$, we have that 
\[
|g(\zeta)|\leq C''e^{\psi_{\alpha}(\zeta)+n\log(1+|\zeta|)}
\leq C'''e^{\psi_{\alpha'}(\zeta)}
\,\,\,\,\,\,\,\,\,\forall\zeta\in\mathbb{C}^{N}
\]
for some $C'',\,C'''>0$ and $\alpha'\in\N$. 
So $g\in{\mathcal O}_{\psi}(\mathbb{C}^{N})$.
\end{proof}

Proposition~\ref{lem:LEMMA3} will be crucial in the study of the homomorphism
(\ref{eq:9}) related to the study of existence and/or uniqueness 
of solutions of the
Cauchy problem (\ref{eq:sistema whitney}). 

To this aim we take $K_{1}$
and $K_{2}$ closed and convex sets, with $\overline{\mathring{K}_{j}}=K_{j}$
for $j=1,\,2$, and $K_{1}\subsetneq K_{2}$.

Then we define, for $j=1,\,2:$
\[
\psi_{\alpha}^{j}(\zeta):=H_{K_{\alpha}^{j}}(\Im\zeta)+\alpha\omega(\zeta)
\]
for $K_{\alpha}^{j}$ compact convex set with 
$K_{\alpha}^{j}\subset\mathring{K}_{\alpha+1}^{j}$
and $\underset{\alpha}{\cup}K_{\alpha}^{j}=K_{j}$, for each $j=1,2$.

We consider the inductive limits
\beqs
\label{G8}
{\mathcal O}_{\psi^{j}}(\mathbb{C}^{N}):=
\indlim_{\alpha\to+\infty}{\mathcal O}_{\psi_{\alpha}^{j}}(\mathbb{C}^{N}),\qquad
j=1,2.
\eeqs

From the above considerations we have the following:
\begin{Rem}
  \begin{em}
The study of the homomorphism \eqref{eq:9} is reduced to the study of 
the homomorphism
\begin{equation}
{\mathcal O}_{\psi^{1}}(V)\rightarrow{\mathcal O}_{\psi^{2}}(V).
\label{eq:14}
\end{equation}
  \end{em}
\end{Rem}

By Proposition~\ref{prop:esistenza} the existence of solutions of the Cauchy
problem \eqref{eq:sistema whitney} is equivalent to the surjectivity
of the homomorphism \eqref{G4}.
But \eqref{G4} has always a dense image, by the following:

\begin{Lemma}
  \label{lemma312}
  Let $\omega\in\W'$, $\wp$ a prime ideal and $K_{1},\,K_{2}$ closed
  convex subsets of $\R^N$ with
$K_{1}\subsetneq K_{2}$, $\overline{\mathring{K}_{j}}=K_{j}$ for
$j=1,\,2$. Then the homomorphism 
\begin{equation}
\Ext_{{\mathcal P}}^{0}\left({\mathcal P}\bigl/\wp,W_{K_{2}}^{(\omega)}\right)
\rightarrow\Ext_{{\mathcal P}}^{0}\left({\mathcal P}\bigl/\wp,W_{K_{1}}^{(\omega)}
\right)
\label{eq:8.1}
\end{equation}
has always a dense image.
\end{Lemma}

\begin{proof}
By Lemma \ref{lem:lemma2} 
\[
\Ext_{{\mathcal P}}^{0}\left({\mathcal P}\bigl/\wp,W_{K_{1}}^{(\omega)}\right)'
\simeq\E_{(\omega)}'(K_{1})\bigl/\wp(D)\otimes\E_{(\omega)}'(K_{1}).
\]

Let $T\in\Ext_{{\mathcal P}}^{0}\left({\mathcal P}\bigl/\wp,W_{K_{1}}^{(\omega)}
\right)'$
which vanish on 
\[
\Ext_{{\mathcal P}}^{0}\left({\mathcal P}\bigl/\wp,W_{K_{2}}^{(\omega)}\right)=
\left\{ u\in W_{K_{2}}^{(\omega)}:\,p_{h}(D)u=0\,\,\,\forall h=1,2,\ldots,r
\right\},
\]
where $p_1(\zeta),\ldots,p_r(\zeta)$ are generators of $\wp$.
We must prove that $T\equiv0$.

By (\ref{eq:16}) we have that 
\[
\hat{T}(\zeta)=T(e^{-i<\cdot,\zeta)})=0\,\,\,\,\,\,\,\,\,\,\forall\zeta\in V,
\]
moreover by the Nullstellensatz (cf. \cite{T}) and the 
Ehrenpreis Fundamental Theorem (cf. \cite{H3}),
\[
T=\sum_{h=1}^{r}p_{h}(D)T_{h}
\]
for some $T_{h}\in\E_{(\omega)}'(K_{1})$, i.e. 
$T\in\wp(D)\otimes\E_{(\omega)}'(K_{1})$.
This shows that $T$ is identically zero as an element of the space
$\E_{(\omega)}'(K_{1})\bigl/\wp(D)\otimes\E_{(\omega)}'(K_{1})\simeq
\Ext_{{\mathcal P}}^{0}\left({\mathcal P}\bigl/\wp,W_{K_{1}}^{(\omega)}\right)'$,
and hence the homomorphism (\ref{eq:8.1}) has a dense image.
\end{proof}

\begin{Rem}
  \label{remG9}
  \begin{em}
    By Proposition~\ref{prop:esistenza} and Lemma~\ref{lemma312},
the Cauchy problem (\ref{eq:sistema whitney}) admits at least
a solution if and only if the homomorphism \eqref{G4} has a closed image,
i.e. if and only if the dual homomorphism \eqref{eq:9} has a closed
image, by \cite[Ch.~IV, \S~2, n.~4, Thm.~3]{Gr}
(see also Remark~\ref{remG11}).

By Proposition~\ref{lem:LEMMA3} we thus have that the Cauchy problem
(\ref{eq:sistema whitney}) admits at least a solution if and only if
the homomorphism (\ref{eq:14}) has a closed image.
  \end{em}
\end{Rem}

By Theorem 5.1 of \cite{BN3} this condition is equivalent
to the validity of the following Phragm\'en-Lindel\"{o}f principle:

\begin{Th}[Phragm\'en-Lindel\"of principle for holomorphic functions]
\label{thm:345} 
Let $V$ be a reduced affine
algebraic variety and ${\mathcal O}_{\psi^{1}}(V)$ and ${\mathcal O}_{\psi^{2}}(V)$
be defined as in \eqref{G8}. Then the following are equivalent:

(i) ${\mathcal O}_{\psi^{1}}(V)\hookrightarrow{\mathcal O}_{\psi^{2}}(V)$
has closed image;

(ii) $\forall\alpha\in\N$, $\exists\,\beta\in\N$ such that 
\[
{\mathcal O}_{\psi^{1}}(V)\cap{\mathcal O}_{\psi_{\alpha}^{2}}(V)
\subset{\mathcal O}_{\psi_{\beta}^{1}}(V);
\]

(iii) the following Phragm\'en-Lindel\"{o}f principle holds:
\begin{align*}
(Ph-L)_{O} & 
\begin{cases}
\forall\alpha\in\N,\,\,\exists\,\beta\in\N,\,\,C>0\,\,
\mbox{such that} & \mbox{}\\
\mbox{if }f\in{\mathcal O}(V)\,\,\mbox{satisfies for some costants }
\alpha_{f}\in\N,\,\,c_{f}>0\\
\begin{cases}
|f(\zeta)|\leq e^{\psi_{\alpha}^{2}(\zeta)}
\,\,\,\,\,\,\,\,\,\,\,\forall\zeta\in V\\
|f(\zeta)|\leq c_{f}e^{\psi_{\alpha_{f}}^{1}(\zeta)}
\,\,\,\,\,\,\,\,\,\,\,\forall\zeta\in V
\end{cases}\\
\mbox{then it also satisfies:}\\
|f(\zeta)|\leq Ce^{\psi_{\beta}^{1}(\zeta)}
\,\,\,\,\,\,\,\,\,\,\,\forall\zeta\in V.
\end{cases}
\end{align*}
\end{Th}

Summarizing, by Remark~\ref{remG9} and Theorem~\ref{thm:345},
we have the following:

\begin{Th}[Phragm\'en-Lindel\"of principle for the existence
of solutions]
  \label{thm:3.4.8}
  Let $\omega\in\W'$.
The Cauchy pro\-blem (\ref{eq:sistema whitney}) admits \textbf{at least}
a solution if and only if the following Phragm\'en-Lindel\"{o}f principle
holds for all $\wp\in\Ass(\M)$ and $V=V(\wp)$ :
\begin{align*}
(Ph-L)_{O} & 
\begin{cases}
\forall\alpha\in\N,\,\,\exists\,\beta\in\N,\,\,C>0\,\,\mbox{such that} 
& \mbox{}\\
\mbox{if }f\in{\mathcal O}(V)\,\,\mbox{satisfies for some costants }
\alpha_{f}\in\N,\,\,c_{f}>0\\
\begin{cases}
|f(\zeta)|\leq\exp\left\{ H_{K_{\alpha}^{2}}(\Im\zeta)+\alpha\omega(\zeta)
\right\} \,\,\,\,\,\,\,\,\,\,\,\forall\zeta\in V\\
|f(\zeta)|\leq c_{f}\exp\left\{ H_{K_{\alpha_{f}}^{1}}(\Im\zeta)+\alpha_{f}
\omega(\zeta)\right\} \,\,\,\,\,\,\,\,\,\,\,\forall\zeta\in V
\end{cases}\\
\mbox{then it also satisfies:}\\
|f(\zeta)|\leq C\exp\left\{ H_{K_{\beta}^{1}}(\Im\zeta)+\beta\omega(\zeta)
\right\} \,\,\,\,\,\,\,\,\,\,\,\forall\zeta\in V.
\end{cases}
\end{align*}
\end{Th}

Let us now recall the definition of plurisubharmonic functions on
an affine algebraic variety $V\subset\mathbb{C}^{N}:$
\begin{Def}
\label{defpsh}
\begin{em}
A function $u:V\rightarrow[-\infty,+\infty)$
is called \emph{plurisubharmonic} on $V$ if it is locally bounded
from above, plurisubharmonic in the usual sense on $V_{{\rm reg}}$,
the set of all regular points of $V$, and satisfies 
\[
u(\zeta)=\underset{}{\underset{z\rightarrow\zeta}{\underset{z\in 
V_{{\rm reg}}}{\lim\,\sup}}}\,u(z)
\]
at the singular points of $V$.

By ${\psh(V)}$ we denote the set of all functions that are plurisubharmonic
on $V$.
\end{em}
\end{Def}

By Theorem 1.2 of \cite{BN2}, Theorem \ref{thm:3.4.8} is equivalent
to the following:

\begin{Th}[Phragm\'en-Lindel\"of principle for
    plurisubharmonic functions]
  Let $\omega\in\W'$.
The Cauchy problem (\ref{eq:sistema whitney}) admits \textbf{at least}
a solution if and only if the following Phragm\'en-Lindel\"{o}f principle
holds for all $\wp\in\Ass(\M)$ and $V=V(\wp)$ :
\begin{align*}
(Ph-L)_{{\psh}} & 
\begin{cases}
\forall\alpha\in\N,\,\,\exists\,\beta\in\N,\,\,C>0\,\,\mbox{such that} 
& \mbox{}\\
\mbox{if }u\in{\psh}(V)\,\,\mbox{satisfies for some costants }\alpha_{u}\in\N,
\,\,c_{u}>0\\
\begin{cases}
u(\zeta)\leq H_{K_{\alpha}^{2}}(\Im\zeta)+\alpha\omega(\zeta)
\,\,\,\,\,\,\,\,\,\,\,\forall\zeta\in V\\
u(\zeta)\leq H_{K_{\alpha_{u}}^{1}}(\Im\zeta)+\alpha_{u}\omega(\zeta)+c_{u}
\,\,\,\,\,\,\,\,\,\,\,\forall\zeta\in V
\end{cases}\\
\mbox{then it also satisfies:}\\
u(\zeta)\leq H_{K_{\beta}^{1}}(\Im\zeta)+\beta\omega(\zeta)+C
\,\,\,\,\,\,\,\,\,\,\,\forall\zeta\in V.
\end{cases}
\end{align*}
\end{Th}

Also the problem of existence of a unique
solution of the Cauchy problem (\ref{eq:sistema whitney}) can be
easily treated by the study of the dual homomorphism (\ref{eq:9}).
In particular, by Propositions \ref{prop:3.3.7} and
\ref{lem:LEMMA3},
we have:

\begin{Th}
  Let $\omega\in\W'$.
The Cauchy problem (\ref{eq:sistema whitney}) admits \textbf{one
and only one} solution if and only if, for all $\wp\in\Ass(\M)$
and $V=V(\wp)$, the homomorphism
\[
{\mathcal O}_{\psi^{1}}(V)\hookrightarrow{\mathcal O}_{\psi^{2}}(V)
\]
is an isomorphism.
\end{Th}

By Theorem 5.2 of \cite{BN3} we can finally state the following:

\begin{Th}
  \label{thm:TEU2}
  Let $\omega\in\W'$.
The Cauchy problem (\ref{eq:sistema whitney}) admits
\textbf{one and only one }solution if and only if, for all $\wp\in\Ass(\M)$
and $V=V(\wp)$, one of the following equivalent conditions holds:

(i) ${\mathcal O}_{\psi^{1}}(V)\hookrightarrow{\mathcal O}_{\psi^{2}}(V)$
is an isomorphism;

(ii) $\forall\alpha\in\N$, $\exists\,\beta\in\N$ such that 
\[
{\mathcal O}_{\psi_{\alpha}^{2}}(V)\subset{\mathcal O}_{\psi_{\beta}^{1}}(V);
\]

(iii) $\forall\alpha\in\N$, $\exists\,\beta\in\N,\,\,C>0$ such that
\[
\underset{\zeta\in V}{\sup}\left|f(\zeta)e^{-H_{K_{\beta}^{1}}(\Im\zeta)
-\beta\omega(\zeta)}\right|\leq C\underset{\zeta\in V}{\sup}
\left|f(\zeta)e^{-H_{K_{\alpha}^{2}}(\Im\zeta)-\alpha\omega(\zeta)}\right|
\]
for all $f\in{\mathcal O}(V)$.
\end{Th}

\begin{Rem}
\begin{em}
Clearly condition \emph{(i) }(resp. \emph{(ii)}, \emph{(iii)}) of
Theorem \ref{thm:TEU2} implies condition \emph{(i)} (resp. \emph{(ii)},
\emph{(iii)}) of Theorem \ref{thm:345}.
\end{em}
\end{Rem}

{\bf Acknowledgments:}
The authors are grateful to Prof. R.~Meise for his helpful suggestions 
about Theorem~\ref{thm:prop 7.2-7.3} and Lemma~\ref{lem:Lemma1}.

The first author is member of the Gruppo Nazionale per l'Analisi Matematica,
la Probabilit\`a e le loro Applicazioni (GNAMPA) of the Istituto Nazionale
di Alta Matematica (INdAM).

\end{document}